\documentclass[11pt,a4paper,final]{article}

\RequirePackage[OT1]{fontenc}
\RequirePackage[colorlinks,citecolor=blue,urlcolor=blue]{hyperref}
\usepackage{amsmath}    
\usepackage{amsfonts}   
\usepackage{amssymb}    
\usepackage{amsthm}     
\usepackage{mathrsfs}   
\usepackage{parskip}    
\usepackage{calc}       
\usepackage{enumitem}   
\usepackage{graphicx}   
\usepackage{lpic}       

\newlength{\vslength}
\newcommand{\ie}{{\it i.e.}}
\newcommand{\cf}{{\it c.f.}}
\newcommand{\eg}{{\it e.g.}}
\newcommand{\ea}{{\it et al.}}
\newcommand{\lhs}{{\it l.h.s.}}
\newcommand{\rhs}{{\it r.h.s.}}
\newcommand{\iid}{{\it i.i.d.}}
\newcommand{\etc}{{\it etcetera}}

\newcommand{\RR}{{\mathbb R}}

\newcommand{\NN}{{\mathbb N}}

\newcommand{\scrB}{{\mathscr B}}

\newcommand{\scrD}{{\mathscr D}}
\newcommand{\scrE}{{\mathscr E}}
\newcommand{\scrF}{{\mathscr F}}
\newcommand{\scrG}{{\mathscr G}}

\newcommand{\scrP}{{\mathscr P}}

\newcommand{\scrS}{{\mathscr S}}
\newcommand{\scrT}{{\mathscr T}}
\newcommand{\scrU}{{\mathscr U}}

\newcommand{\scrX}{{\mathscr X}}

\newcommand{\samplen}{{X^{n}}}
\newcommand{\sample}{{X}}
\newcommand{\realizationn}{{x^{n}}}

\newcommand{\al}{{\alpha}}
\newcommand{\ep}{{\epsilon}}
\newcommand{\vep}{{\varepsilon}}
\newcommand{\tht}{{\theta}}

\newcommand{\Tht}{{\Theta}}
\renewcommand{\emptyset}{{\varnothing}}

\newcommand{\ctg}{\mathbin{\lhd}}

\newcommand{\ft}[2]{{\textstyle{\frac{#1}{#2}}}}

\newcommand{\conv}[1]%
  {{\mathrel{\,\xrightarrow{\widthof{\,#1\,}}\,}}}
\newcommand{\convas}[1]%
  {{\mathrel{\,\xrightarrow{\widthof{\,#1\text{-a.s.}\,}}\,}}}
\newcommand{\convprob}[1]%
  {{\mathrel{\,\xrightarrow{\widthof{\,#1\,}}\,}}}
\newcommand{\convweak}[1]%
  {{\mathrel{\,\xrightarrow{\widthof{\,#1\text{-w.}\,}}\,}}}

\newcommand{\twobytwo}[4]%
  {\left(\begin{array}{cc} #1 & #2 \\ #3 & #4 \end{array}\right)}
\newcommand{\twovec}[2]%
  {\left({\begin{array}{c} #1\\#2 \end{array}}\right)}
\newcommand{\contig}{\mathop{\vartriangleleft}}

\renewcommand{\qedsymbol}{$\Box$}
\newcommand{\closebox}{\hfill\qedsymbol}

\newtheoremstyle{customtheorem}
  {0.5em}
  {0.2em}
  {\itshape}
  {}
  {\scshape}
  {}
  {1ex}
  {}

\theoremstyle{customtheorem}
\newtheorem{theorem}{Theorem}[section]
\newtheorem{lemma}[theorem]{Lemma}
\newtheorem{proposition}[theorem]{Proposition}
\newtheorem{corollary}[theorem]{Corollary}
\newtheorem{definition}[theorem]{Definition}

\newtheoremstyle{customremark}
  {0.5em}
  {0.2em}
  {}
  {}
  {\scshape}
  {}
  {1ex}
  {}

\theoremstyle{customremark}
\renewenvironment{proof}{\par\noindent{\scshape Proof}\;}{\hfill\qedsymbol\par}
\newtheorem{remark}[theorem]{Remark}
\newtheorem{example}[theorem]{Example}

\newcommand{\comment}[1]{{}}
\newcommand{\extra}[1]{{}}

\begin{document}

\thispagestyle{empty}

\title{\vspace*{-11mm}
  On the frequentist validity of Bayesian limits}
\author{
  B.~J.~K.~{Kleijn}\footnote{email: B.Kleijn@uva.nl, web: https://staff.fnwi.uva.nl/b.j.k.kleijn}\\[1mm]
  {\small\it Korteweg-de~Vries Institute for Mathematics,
    University of Amsterdam}\\[2mm]
  }
\date{November 2017}
\maketitle

\begin{abstract}
To the frequentist who computes posteriors, not
all priors are useful asymptotically: in this
paper Schwartz's 1965 Kullback-Leibler condition
is generalised to enable frequentist interpretation
of convergence of posterior distributions with the
complex models and often dependent datasets in
present-day statistical applications. We prove
four simple and fully general frequentist theorems,
for posterior consistency; for posterior rates of
convergence; for consistency of the Bayes factor
in hypothesis testing or model selection; and a
theorem to obtain confidence sets from credible
sets. The latter has a significant methodological
consequence in frequentist uncertainty
quantification: use of a suitable prior allows
one to convert credible sets of a calculated,
simulated or approximated posterior into
asymptotically consistent confidence sets, in
full generality. This extends the main
inferential implication of the Bernstein-von~Mises
theorem to non-parametric models without
smoothness conditions. Proofs require the
existence of a Bayesian type of test sequence
and priors giving rise to local prior predictive
distributions that satisfy a weakened form of
Le~Cam's contiguity with respect to the data
distribution. Results are applied in a wide range 
0f examples and counterexamples.
\end{abstract}


\section{Introduction}
\label{sec:intro}

In this paper (following \cite{Bayarri04}:{\it ``Statisticians should
readily use both Bayesian and frequentist ideas.''}) we examine
for which priors Bayesian asymptotic conclusions extend to conclusions
valid in the frequentist sense: how Doob's prior-almost-sure
consistency is strengthened to reach Schwartz's frequentist conclusion that the
posterior is consistent, or how a test that is consistent prior-almost-surely
becomes a test that is consistent in \emph{all} points of the model,
or how a Bayesian credible set can serve as a frequentist confidence set
asymptotically.

The central property to enable frequentist interpretation
of posterior asymptotics is defined as \emph{remote contiguity} in
section~\ref{sec:rc}. It expresses a weakened form of Le~Cam's contiguity,
relating the true distribution of the data to localized prior predictive
distributions. Where Schwartz's Kullback-Leibler neighbourhoods represent a
choice for the localization appropriate when the sample is \iid,
remote contiguity
generalises the notion to include non-\iid\ samples, priors that change with
the sample size, weak consistency with the Dirichlet process prior, \etc.

Although firstly aimed at enhancing insight into asymptotic relations
by simplification and generalisation, this paper also has
a significant methodological consequence: theorem~\ref{thm:coverage}
demonstrates that if the prior is such that remote contiguity applies,
credible sets can be converted to asymptotically consistent confidence
sets in full generality. So the asymptotic validity of credible sets as
confidence sets in smooth parametric models \cite{LeCam90} extends much
further: in practice, the frequentist can simulate the posterior
in any model, construct his preferred type of credible sets and `enlarge'
them to obtain asymptotic confidence sets, provided his prior
induces remote contiguity. This extends the
main inferential implication of the Bernstein-von Mises theorem to
non-parametric models.

In the remainder of this section we discuss posterior consistency.
In section~\ref{sec:post} we concentrate on an inequality that relates
testing to posterior concentration and indicates the relation with Le~Cam's inequality. Section~\ref{sec:rc} introduces remote contiguity and
the analogue of Le~Cam's First Lemma. In section~\ref{sec:four}, 
frequentist theorems on the asymptotic behaviour of posterior distributions
are proved, on posterior consistency, on posterior rates of convergence,
on consistent testing and model selection with Bayes factors and on
the conversion of credible sets to confidence sets. Section~\ref{sec:disc}
formulates the conclusions.

Definitions, notation, conventions roughly follow those of \cite{LeCam86}
and are collected in appendix~\ref{sec:defs} with some other
preliminaries. 
All applications, illustrations, examples and counterexamples have been collected in
appendix~\ref{sec:app}. Proofs are found in appendix~\ref{sec:proofs}.

\subsection{Posterior consistency and inconsistency}
\label{sub:sketch}

For a statistical procedure to be \emph{consistent}, it must infer 
the truth with arbitrarily large accuracy and probability, if we
gather enough data.
For example, when using sequential data $X^n\sim P_{\tht_0,n}$ to
estimate the value $\tht_0$, a consistent estimator sequence
$\tht_n$ converges to $\tht_0$ in $P_{\tht_0,n}$-probability.
For a posterior $\Pi(\cdot|X^n)$ to be consistent,
it must concentrate mass arbitrarily close to one in any neighbourhood
of $\tht_0$ as $n\to\infty$ (see definition~\ref{def:consistency}).

Consider a model $\scrP$ for \iid\ data with single-observation
distribution $P_0$. Give $\scrP$ a Polish topology with Borel
prior $\Pi$ so that the posterior is well-defined
(see definition~\ref{def:posterior}). The first
general consistency theorem for posteriors is due to Doob.
\begin{theorem}
\label{thm:doob}{\it (Doob (1949))}\\
For all $n\geq1$, let $(X_1, X_2, \ldots, X_n)\in\scrX^n$ be
$\iid-P_0$, where $P_0$ lies in a model $\scrP$. Suppose
$\scrX$ and $\scrP$ are Polish spaces. Assume that $P\mapsto P(A)$ is
Borel measurable for every Borel set $A\subset\scrX$. Then for any
Borel prior $\Pi$ on $\scrP$ the posterior is consistent, for
$\Pi$-almost-all $P$.
\end{theorem}
In parametric applications Doob's $\Pi$-null-set of potential
inconsistency can be considered small (for example, when the
prior dominates Lebesgue measure). But in non-parametric context these
null-sets can become very large (or not, see \cite{Lijoi04}): the first
examples of unexpected posterior inconsistency are due to Schwartz
\cite{Schwartz61}, but it was Freedman \cite{Freedman63} who made
the point famous with a simple non-parametric counterexample
(discussed in detail as example~\ref{ex:freedman63}). In
\cite{Freedman65} it was even shown that inconsistency is generic
in a topological sense:
the set of pairs $(P_0,\Pi)$ for which the posterior is consistent
is \emph{meagre}: posteriors that only wander around, placing
and re-placing mass aimlessly, are the rule rather than the exception.
(For a discussion, see example~\ref{ex:freedman65}.)

These and subsequent examples of posterior
inconsistency established a widespread conviction
that Bayesian methods were wholly unfit for
frequentist purposes, at least in non-parametric context. The only
justifiable conclusion from Freedman's meagreness, however, is that
a condition is missing: Doob's assertion may be all that a Bayesian
requires, a frequentist demands strictly more, thus restricting the
class of possible choices for his prior. 
Strangely, a condition representing this restriction had already
been found when Freedman's meagreness result was published.
\begin{theorem}
\label{thm:schwartz}{\it (Schwartz (1965))}\\
For all $n\geq1$, let $(X_1, X_2, \ldots, X_n)\in\scrX^n$ be
$\iid-P_0$, where $P_0$ lies in a model $\scrP$. Let $U$ denote
an open neighbourhood of $P_0$ in $\scrP$. If,
\begin{itemize}[noitemsep,nolistsep]
\item[(i)] there exist measurable $\phi_n:\scrX^n\to[0,1]$, such that,  
  \begin{equation}
    \label{eq:schwartzuniformtests}
    P_0^n\phi_n =o(1),\quad \sup_{Q\in U^c}Q^n(1-\phi_n)=o(1),
  \end{equation}
\item[(ii)] and $\Pi$ is a Kullback-Leibler prior, \ie\ for all $\delta>0$,  
  \begin{equation}
    \label{eq:KLprior}
    \Pi\Bigl(\,P\in\scrP\,:\,
      -P_0\log\frac{dP}{dP_0}<\delta\,\Bigr) > 0,
  \end{equation}
\end{itemize}
then $\Pi(U|\samplen)\convas{P_0}1$.
\end{theorem}
Over the decades, examples of problematic posterior behaviour in
non-parametric setting continued to captivate 
\cite{Diaconis86a, Diaconis86b, Cox93, Diaconis93, Diaconis98,
Freedman83, Freedman99}, while Schwartz's theorem received
initially limited but steadily growing amounts of attention:
subsequent frequentist theorems (\eg\ by Barron \cite{Barron88},
Barron-Schervish-Wasserman \cite{Barron99},
Ghosal-Ghosh-van~der~Vaart \cite{Ghosal00},
Shen-Wasserman \cite{Shen01}, Walker \cite{Walker04}
and Walker-Lijoi-Pr\"unster \cite{Walker07}, Kleijn-Zhao
\cite{Kleijn16} and many others) have extended the applicability
of theorem~\ref{thm:schwartz} but not its essence, condition
(\ref{eq:KLprior}) for the prior. The following example illustrates
that Schwartz's condition cannot be the whole truth, though.
\begin{example}
\label{ex:noKLpriors}
Consider $X_1,X_2,\ldots$ that are \iid-$P_0$
with Lebesgue density $p_0:\RR\rightarrow\RR$ supported on
an interval of known width (say, $1$) but unknown location.
Parametrize in terms of a continuous density $\eta$ on
$[0,1]$ with $\eta(x)>0$ for all $x\in[0,1]$ and a location
$\tht\in\RR$:
$p_{\tht,\eta}(x) = \eta(x-\tht)\,1_{ [\tht,\tht+1] }(x)$.
A moment's thought makes clear that if $\tht\neq\tht'$,
\[
  -P_{\tht,\eta}\log\frac{p_{\tht',\eta'}}{p_{\tht,\eta}} = \infty,
\]
for all $\eta,\eta'$. Therefore Kullback-Leibler neighbourhoods
do not have any extent in the $\tht$-direction and \emph{no prior is
a Kullback-Leibler prior in this model}. Nonetheless the posterior
is consistent (see examples~\ref{ex:rcnoKLpriors}
and~\ref{ex:testnoKLpriors}).\closebox
\end{example}
Similar counterexamples exist \cite{Kleijn16} for the type of
prior that is proposed in the analyses of posterior rates of
convergence in (Hellinger) metric setting \cite{Ghosal00,Shen01}.
Although methods in \cite{Kleijn16} avoid this type of problem,
the essential nature of condition (\ref{eq:KLprior}) in \iid\
setting becomes apparent there as well.

This raises the central question of this paper: is Schwartz's
Kullback-Leibler condition perhaps a manifestation of a
more general notion? The argument leads to other questions
for which insightful answers have been elusive: why is Doob's
theorem completely different from Schwartz's? The accepted
explanation views the lack of congruence as an indistinct
symptom of differing philosophies, but is this justified?
Why does weak consistency in the full non-parametric model
(\eg\ with the Dirichlet process prior \cite{Ferguson73}, or
more modern variations \cite{DeBlasi13}) reside in a corner of
its own (with \emph{tailfreeness} \cite{Freedman65} as
sufficient property of the prior), apparently unrelated to
posterior consistency in either Doob's or Schwartz's views?
Indeed, what would Schwartz's theorem look like without the
assumption that the sample is \iid\ (\eg\ with data that form
a Markov chain or realize some other stochastic process)
or with growing parameter spaces and changing priors?
And to extend the scope further, what
can be said about hypothesis testing, classification,
model selection, \etc? Given that the Bernstein-von~Mises
theorem cannot be expected to hold in any generality outside
parametric setting \cite{Cox93,Freedman99}, what relationship
exists between credible sets and confidence sets? This paper
aims to shed more light on these questions in a general sense,
by providing a prior condition that enables strengthening
Bayesian asymptotic conclusions to frequentist ones, illustrated
with a variety of examples and counterexamples.


\section{Posterior concentration and asymptotic tests}
\label{sec:post}

In this section, we consider a lemma that relates concentration of
posterior mass in certain model subsets to the existence of
test sequences that distinguish between those subsets.
More precisely, it is shown that the expected posterior mass outside a
model subset $V$ with respect to the local prior predictive
distribution over a model subset $B$, is upper bounded (roughly)
by the testing power of \emph{any} statistical test for the hypotheses
$B$ versus $V$: if a test sequence exists, the posterior will
concentrate its mass appropriately.

\subsection{Bayesian test sequences}
\label{sub:tests}

Since the work of Schwartz \cite{Schwartz65}, test sequences and
posterior convergence have been linked intimately. Here we
follow Schwartz and consider asymptotic testing; however, we define
test sequences immediately in Bayesian context by involving priors
from the outset. 
\begin{definition}
Given priors $(\Pi_n)$, measurable model
subsets $(B_n),(V_n)\subset\scrG$ and $a_n\downarrow0$,
a sequence of $\scrB_n$-measurable maps
$\phi_n:\scrX_n\rightarrow[0,1]$ is called a 
\emph{Bayesian test sequence for $B_n$ versus $V_n$ (under $\Pi_n$)
of power $a_n$}, if,
\begin{equation}
  \label{eq:bayestest}
  \int_{B_n} P_{\tht,n}\phi_n\,d\Pi_n(\tht) +
  \int_{V_n} P_{\tht,n}(1-\phi_n)\,d\Pi_n(\tht) = o(a_n).
\end{equation}
We say that $(\phi_n)$ is a \emph{Bayesian test sequence for $B_n$
versus $V_n$ (under $\Pi_n$)} if (\ref{eq:bayestest}) holds for
some $a_n\downarrow0$.
\end{definition}

Note that if we have sequences $(C_n)$ and $(W_n)$ such that
$C_n\subset B_n$ and $W_n\subset V_n$ for all $n\geq1$, then a
Bayesian test sequence for $(B_n)$ versus $(V_n)$  of
power $a_n$ is a Bayesian test sequence for $(C_n)$ versus $(W_n)$
of power (at least) $a_n$.
\begin{lemma}
\label{lem:testineq}
For any $B,V\in\scrG$ with $\Pi(B)>0$ and any measurable
$\phi:\scrX\rightarrow[0,1]$,
\begin{equation}
\label{eq:testineq}
  \int P_\tht\Pi(V|X)\,d\Pi(\tht|B)
  \leq \int P_\tht\phi\,d\Pi(\tht|B)
    + \frac{1}{\Pi(B)} \int_V P_{\tht}(1-\phi)\,d\Pi(\tht).
\end{equation}
\end{lemma}
So the mere existence of a test sequence is enough to guarantee
posterior concentration, a fact expressed in $n$-dependent form
through the following proposition.
\begin{proposition}
\label{prop:prototype}
Assume that for given priors $\Pi_n$, sequences $(B_n),(V_n)\subset\scrG$
and $a_n,b_n\downarrow0$ such that $a_n=o(b_n)$ with $\Pi_n(B_n)\geq b_n>0$,
there exists a Bayesian test sequence for $B_n$ versus $V_n$ of power $a_n$.
Then,
\begin{equation}
  \label{eq:prototype} 
  P_n^{\Pi_n|B_n}\Pi(V_n|\samplen) = o(a_n\,b_n^{-1}),
\end{equation}
 for all $n\geq1$.
\end{proposition}
To see how this leads to posterior consistency, consider the following: if
the model subsets $V_n=V$ are all equal to the complement of a neighbourhood
$U$ of $P_0$, and the $B_n$ are chosen such that the expectations of the random
variables $\samplen\mapsto\Pi(V|\samplen)$ under $P_n^{\Pi_n|B_n}$
`dominate' their expectations under $P_{0,n}$ in a suitable way, 
sufficiency of prior mass $b_n$ given testing power $a_n\downarrow0$,
is enough to assert that $P_{0,n}\Pi(V|\samplen)\rightarrow0$, so
an arbitrarily large fraction of posterior mass is found in $U$ with
high probability for $n$ large enough.

\subsection{Existence of Bayesian test sequences}
\label{sub:existencebayestests}

Lemma~\ref{lem:testineq} and proposition~\ref{prop:prototype} require
the existence of test sequences of the Bayesian type. That question
is unfamiliar, frequentists are used to test sequences for
pointwise or uniform testing. For example, an application of
Hoeffding's inequality demonstrates that,
weak neighbourhoods are uniformly testable (see
proposition~\ref{prop:weaktests}). Another well-known example
concerns testability of convex model subsets. Mostly the uniform
test sequences in Schwartz's theorem are constructed using convex
building blocks $B$ and $V$ separated in Hellinger distance
(see proposition~\ref{prop:minmaxhell} and subsequent remarks).

Requiring the existence of a Bayesian test sequence
\cf\ (\ref{eq:bayestest}) is quite different. We shall
illustrate this point in various ways below.
First of all the existence of a Bayesian
test sequence is linked directly to behaviour of the posterior
itself.
\begin{theorem}
\label{thm:testconsequi}
Let $(\Tht,\scrG,\Pi)$ be given. For any $B,V\in\scrG$ with
$\Pi(B)>0, \Pi(V)>0$, the following are equivalent,
\begin{itemize}[noitemsep,nolistsep]
\item[(i)] there are $\scrB_n$-measurable $\phi_n:\scrX_n\to[0,1]$
  such that for $\Pi$-almost-all $\tht\in B, \tht'\in V$,
  \[
    P_{\tht,n}\phi_n\to0,\quad P_{\tht',n}(1-\phi_n)\to0,
  \]
\item[(ii)] there are $\scrB_n$-measurable $\phi_n:\scrX_n\to[0,1]$
  such that,
  \[
    \int_B P_{\tht,n}\phi_n\,d\Pi(\tht) +
      \int_V P_{\tht,n}(1-\phi_n)\,d\Pi(\tht)\to 0,
  \]
\item[(iii)] for $\Pi$-almost-all $\tht\in B$, $\tht'\in V$,
\[
  \Pi(V|\samplen)\conv{P_{\tht,n}}0,\qquad\Pi(B|\samplen)\conv{P_{\tht',n}}0.
\]
\end{itemize}
\end{theorem}
The interpretation of this theorem is gratifying to supporters of
the likelihood principle and pure Bayesians: distinctions between
model subsets are Bayesian testable, if and only if, they are picked
up by the posterior asymptotically, if and only if, there exists a
pointwise test for $B$ versus $V$ that is $\Pi$-almost-surely consistent.

For a second, more frequentist way to illustrate how basic the
existence of a Bayesian test sequences is, consider a
parameter space $(\Tht,d)$ which is a metric space with fixed Borel
prior $\Pi$ and $d$-consistent estimators $\hat\tht_n:\scrX_n\to\Tht$
for $\tht$. Then for every $\tht_0\in\Tht$ and
$\ep>0$, there exists a pointwise test sequence (and hence, by
dominated convergence, also a Bayesian test sequence)
for $B=\{\tht\in\Tht:d(\tht,\tht_0)<\ft12\ep\}$ versus
$V=\{\tht\in\Tht:d(\tht,\tht_0)>\ep\}$. This approach is followed
in example~\ref{ex:gofmarkov} on random walks, see the definition
of the test following inequality (\ref{eq:hoeff}).

A third perspective on the existence of Bayesian tests arises from
Doob's argument. From our present perspective, we note that
theorem~\ref{thm:testconsequi} implies an alternative proof of
Doob's consistency theorem through the following existence
result on Bayesian test sequences. (Note: here and elsewhere in
\iid\ setting, the parameter space $\Tht$ is $\scrP$, $\tht$ is
the single-observation distribution $P$ and $\tht\mapsto P_{\tht,n}$ is
$P\mapsto P^n$.)
\begin{proposition}
\label{prop:msbtest}
Consider a model $\scrP$ of single-observation distributions $P$
for \iid\ data $(X_1,X_2,\ldots,X_n)\sim P^n$, $(n\geq1)$.
Assume that $\scrP$ is a Polish space with Borel prior $\Pi$.
For any Borel set $V$ there is a Bayesian test sequence for $V$
versus $\scrP\setminus V$ under $\Pi$.
\end{proposition}
Doob's theorem is recovered when we let $V$ be the complement of
any open neighbourhood $U$ of $P_0$.
Comparing with conditions for the existence of uniform tests,
Bayesian tests are quite abundant: whereas uniform testing relies
on the minimax theorem (forcing convexity, compactness and
continuity requirements into the picture), Bayesian tests exist
quite generally (at least, for Polish parameters with \iid\ data).

The fourth perspective on the existence of Bayesian tests concerns
a direct way to construct a Bayesian test sequence of optimal
power, based on the fact that we are really only testing
barycentres against each other: let priors $(\Pi_n)$ and
$\scrG$-measurable
model subsets $B_n,V_n$ be given. For given tests $(\phi_n)$ and
power sequence $a_n$, write (\ref{eq:bayestest}) as follows:
\[
  \Pi_n(B_n)\,P_{n}^{\Pi_n|B_n}\phi_n(\samplen)
    + \Pi_n(V_n)\,P_{n}^{\Pi_n|V_n}\phi_n(\samplen) = o(a_n),
\]
and note that what is required here, is a (weighted)
test of $(P_{n}^{\Pi_n|B_n})$ versus $(P_{n}^{\Pi_n|V_n})$. The
likelihood-ratio test (denote the density for $P_{n}^{\Pi_n|B_n}$
with respect to $\mu_n = P_{n}^{\Pi_n|B_n} + P_{n}^{\Pi_n|V_n}$ by
$p_{B_n,n}$, and similar for $P_{n}^{\Pi_n|V_n}$),
\[
  \phi_n(\samplen) = 1_{\{\Pi_n(V_n)\,p_{V_n,n}(\samplen)
    > \Pi_n(B_n)\,p_{B_n,n}(\samplen)\}},
\]
is optimal and has power
$\| \Pi_n(B_n)\,P_{n}^{\Pi_n|B_n} \mathop{\wedge}
\Pi_n(V_n)\,P_{n}^{\Pi_n|B_n} \|$. This proves the following
useful proposition that re-expresses
power in terms of the relevant Hellinger transform (see, \eg\
section~16.4 in \cite{LeCam86}, particularly, Remark~1).
\begin{proposition}
\label{prop:barycentres}
Let priors $(\Pi_n)$ and measurable model subsets $B_n,V_n$ be
given. There exists a test sequence $\phi_n:\scrX_n\to[0,1]$
such that,
\begin{equation}
\label{eq:bayespower}
\begin{split}
  \int_{B_n} &P_{\tht,n}\phi_n\,d\Pi_n(\tht) +
  \int_{V_n} P_{\tht,n}(1-\phi_n)\,d\Pi_n(\tht)\\
  &\leq
    \int \Bigl( \Pi_n(B_n)\,p_{B_n,n}(x)\Bigr)^\al
      \Bigl( \Pi_n(V_n)\,p_{V_n,n}(x) \Bigr)^{1-\al}\,d\mu_n(x),
\end{split}
\end{equation}
for every $n\geq1$ and any $0\leq\al\leq1$.
\end{proposition}
Proposition~\ref{prop:barycentres} generalises
proposition~\ref{prop:msbtest} and makes Bayesian
tests available with a (close-to-)sharp bound on the power under
fully general conditions. For the connection with minimax tests,
we note the following. If $\{P_{\tht,n}:\tht\in B_n\}$
and $\{P_{\tht,n}:\tht\in V_n\}$ are
\emph{convex} sets (and the $\Pi_n$ are Radon measures, \eg\
in Polish parameter spaces), then,
\[
  H\bigl(P_{n}^{\Pi_n|B_n},P_{n}^{\Pi_n|V_n}\bigr)
    \geq \inf\{ H(P_{\tht,n},P_{\tht',n}):
      \tht\in B_n, \tht'\in V_n \}.
\]
Combination with (\ref{eq:bayespower}) for $\al=1/2$,
implies that the minimax upper bound in
\iid\ cases, \cf\ proposition~\ref{prop:minmaxhell} remains
valid:
\begin{equation}
  \label{eq:bayeshelltests}
  \int_{B_n} P^n\phi_n\,d\Pi_n(P) +
  \int_{V_n} Q^n(1-\phi_n)\,d\Pi_n(Q)
  \leq
  \sqrt{\Pi_n(B_n)\,\Pi_n(V_n)}\,e^{-n\ep_n^2},
\end{equation}
where $\ep_n=\inf\{ H(P,Q): P\in B_n, Q\in V_n \}$.
Given $a_n\downarrow0$, any Bayesian test $\phi_n$ that satisfies
(\ref{eq:bayestest}) \emph{for all} probability measures $\Pi_n$ on
$\Tht$, is a (weighted) minimax test for $B_n$ versus $V_n$ of
power $a_n$.

Note that the above enhances the role that the prior plays in the
frequentist discussion of the asymptotic behaviour of the
posterior: the prior is not only important in requirements
like (\ref{eq:KLprior}), 
but can also be of influence in the testing condition: where
testing power is relatively weak, prior mass should be scarce
to compensate and where testing power is strong, prior mass
should be plentiful. To make use of this, one typically imposes
\emph{upper} bounds on prior mass in certain hard-to-test
subsets of the model (as opposed to \emph{lower} bounds like
(\ref{eq:KLprior})). See example~\ref{ex:gofmarkov} on
random-walk data. In the Hellinger-geometric view, the prior
determines whether the local prior predictive distributions
$P_{n}^{\Pi_n|B_n}$ and $P_{n}^{\Pi_n|V_n}$ lie close together
or not in Hellinger distance, and thus to
the \rhs\ of (\ref{eq:bayespower}) for $\al=1/2$. This
phenomenon plays a role in example~\ref{ex:sparsenormalmeans} on the
estimation of a sparse vector of normal means, where it explains why the
\emph{slab}-component of a \emph{spike-and-slab prior} must have
a tail that is heavy enough.

\subsection{Le~Cam's inequality}

Referring to the argument following proposition~\ref{prop:prototype},
one way of guaranteeing that the expectations of
$\samplen\mapsto\Pi(V|\samplen)$ under $P_n^{\Pi|B_n}$ approximate
those under $P_{0,n}$, is to choose
$B_n=\{\tht\in\Tht:\|P_{\tht,n}-P_{\tht_0,n}\|\leq \delta_n\}$,
for some sequence $\delta_n\rightarrow0$, because in that case,
$|P_{0,n}\psi - P_n^{\Pi|B_n}\psi| \leq \| P_{0,n}-P_n^{\Pi|B_n} \|
\leq \delta_n$, for any random variable $\psi:\scrX_n\rightarrow[0,1]$.
Without fixing the definition of the sets $B_n$, one may use this step
to specify inequality (\ref{eq:testineq}) further:
\begin{equation}
  \label{eq:lecamineq}
  \begin{split}
  P_{0,n}\Pi(&V_n|X) \leq \bigl\| P_{0,n}-P_n^{\Pi|B_n}\bigr\|\\
    &+ \int P_{\tht,n}\phi_n\,d\Pi_n(\tht|B_n)
    + \frac{\Pi_n(V_n)}{\Pi_n(B_n)} \int P_{\tht,n}(1-\phi_n)\,
      d\Pi_n(\tht|V_n),
  \end{split}
\end{equation}
for $B_n$ and $V_n$ such that $\Pi_n(B_n)>0$ and
$\Pi_n(V_n)>0$. Le~Cam's inequality~(\ref{eq:lecamineq}) is used, for example,
in the proof of the Bernstein-von~Mises theorem, see lemma~2 in section~8.4
of \cite{LeCam90}. A less successful application pertains to non-parametric
posterior rates of convergence for \iid\ data, in an unpublished paper
\cite{LeCam7X}. Rates of convergence obtained in this way
are suboptimal: Le~Cam qualifies the first term on the right-hand side of
(\ref{eq:lecamineq}) as a {\it ``considerable nuisance''} and concludes that
{\it ``it is unclear at the time of this writing what general features,
besides the metric structure, could be used
to refine the results''}, (see \cite{LeCam86}, end of section~16.6). In
\cite{Yang98}, Le~Cam relates the posterior question to dimensionality
restrictions \cite{LeCam73,Shen01,Ghosal00} and reiterates, {\it ``And for
Bayes risk, I know that just the metric structure does not catch
everything, but I don't know what else to look at, except calculations.''}

\section{Remote contiguity}
\label{sec:rc}

Le~Cam's notion of contiguity describes an asymptotic version of
absolute continuity, applicable to sequences of probability measures
in a limiting sense \cite{LeCam60b}.
In this section we weaken the property of contiguity
in a way that is suitable to promote $\Pi$-almost-everywhere Bayesian
limits to frequentist limits that hold everywhere.

\subsection{Definition and criteria for remote contiguity}
\label{sub:rc}

The notion of `domination' left undefined in the argument following
proposition~\ref{prop:prototype} is made rigorous here.
\begin{definition}
\label{def:remctg}
Given measurable spaces $(\scrX_n,\scrB_n)$, $n\geq1$ with two
sequences $(P_n)$ and $(Q_n)$ of probability measures and a sequence
$\rho_n\downarrow0$, we say that
$Q_n$ is $\rho_n$-remotely contiguous with respect to $P_n$,
notation $Q_n\contig \rho_n^{-1}P_n$,
if,
\begin{equation}
  \label{eq:defrc}
  P_n\phi_n(\samplen) = o(\rho_n)
  \quad\Rightarrow\quad Q_n\phi_n(\samplen)=o(1),
\end{equation}
for every sequence of $\scrB_n$-measurable $\phi_n:\scrX_n\rightarrow[0,1]$.
\end{definition}
Note that for a sequence $(Q_n)$ that
is $a_n$-remotely contiguous with respect to $(P_n)$, there exists
no test sequence that distinguishes between $P_n$
and $Q_n$ with power $a_n$. Note also that given two sequences $(P_n)$
and $(Q_n)$, contiguity $P_n\ctg Q_n$ is equivalent to remote contiguity
$P_n\ctg a_n^{-1} Q_n$ for all $a_n\downarrow0$.
Given sequences $a_n,b_n\downarrow0$ with $a_n=O(b_n)$, $b_n$-remote
contiguity implies $a_n$-remote contiguity of $(P_n)$ with respect
to $(Q_n)$.
\begin{example}
\label{ex:KLclose}
Let $\scrP$ be a model for the distribution
of a single observation in \iid\ samples
$\samplen=(X_1,\ldots,X_n)$. Let $P_0, P$ and $\ep>0$ be such
that $-P_0\log(dP/dP_0)<\ep^2$. The law of large numbers
implies that for large enough $n$,
\begin{equation}
  \label{eq:lowerbndlik}
  \frac{dP^n}{dP_0^n}(\samplen)\geq e^{-\frac{n}{2}\ep^2}, 
\end{equation}
with $P_0^n$-probability one.
Consequently, for large enough $n$ and for any $\scrB_n$-measurable
sequence $\psi_n:\scrX_n\rightarrow[0,1]$, 
\begin{equation}
  \label{eq:rcexp}
  P^n\psi_n \geq e^{-\frac{1}{2}n\ep^2}P_0^n\psi_n.
\end{equation}
Therefore, if $P^n\phi_n=o(\exp{(-\ft12n\ep^2)})$
then $P_0^n\phi_n=o(1)$. Conclude that for every $\ep>0$, the
Kullback-Leibler neighbourhood $\{P:-P_0\log(dP/dP_0)<\ep^2\}$
consists of model distributions for which the sequence $(P_0^n)$ of
product distributions are $\exp{(-\ft12n\ep^2)}$-remotely contiguous
with respect to $(P^n)$.\closebox
\end{example}

Criteria for remote contiguity are given in the lemma below;
note that, here, we give sufficient conditions, rather than
necessary and sufficient, as in Le~Cam's First Lemma. (For the
definition of $(dP_n/dQ_n)^{-1}$, see appendix~\ref{sec:defs},
{\it notation and conventions}.) 
\begin{lemma}
\label{lem:rcfirstlemma}
Given $(P_n)$, $(Q_n)$, $a_n\downarrow0$,
$Q_n\contig a_n^{-1}P_n$, if any of the following hold:
\begin{itemize}
\item[(i)] for any $\scrB_n$-measurable
  $\phi_n:\scrX_n\rightarrow[0,1]$,
  $a_n^{-1}\phi_n\conv{P_n}0$ implies $\phi_n\convprob{Q_n}0$,
\item[(ii)] given $\ep>0$, there is a $\delta>0$ such that
  $Q_n(dP_n/dQ_n<\delta\,a_n)<\ep$, for large enough $n$,
\item[(iii)] there is a $b>0$ such that
  $\liminf_{n} b\,a_n^{-1}P_n(dQ_n/dP_n>b\,a_n^{-1})=1$,
\item[(iv)] for any $\ep>0$, there is a constant $c>0$ such that
  $\|Q_n-Q_n\wedge c\,a_n^{-1}P_n\|<\ep$, for large enough $n$,
\item[(v)] 
  under $Q_n$ every subsequence of $(a_n(dP_n/dQ_n)^{-1})$
  has a weakly convergent subsequence.
\end{itemize}
\end{lemma}
\begin{proof}
The proof of this lemma can be found in appendix~\ref{sec:proofs}. It
actually proves that
({\it (i)} or {\it (iv)}) implies remote contiguity;
that ({\it (ii)} or {\it (iii)}) implies {\it (iv)} 
and that {\it (v)} is equivalent to {\it (ii)}.
\end{proof}
Contiguity and its remote variation are compared in the context of
(parametric and non-parametric) regression in
examples~\ref{ex:regression} and~\ref{ex:rcLAN}.
We may specify the definition of remote contiguity slightly further.
\begin{definition}
\label{def:specremctg}
Given measurable spaces $(\scrX_n,\scrB_n)$, $(n\geq1)$ with two
sequences $(P_n)$ and $(Q_n)$ of probability measures and sequences
$\rho_n,\sigma_n>0$, $\rho_n,\sigma_n\rightarrow0$,
we say that
$Q_n$ is $\rho_n$-to-$\sigma_n$ remotely contiguous with respect to $P_n$,
notation $\sigma_n^{-1}Q_n\contig \rho_n^{-1}P_n$,
if,
\[
  P_n\phi_n(\samplen) = o(\rho_n)
  \quad\Rightarrow\quad Q_n\phi_n(\samplen)=o(\sigma_n),
\]
for every sequence of $\scrB_n$-measurable $\phi_n:\scrX_n\rightarrow[0,1]$.
\end{definition}
Like definition~\ref{def:remctg}, definition~\ref{def:specremctg} allows
for reformulation similar to lemma~\ref{lem:rcfirstlemma}, \eg\ if for
some sequences $\rho_n,\sigma_n$ like in definition~\ref{def:specremctg},
\[
  \bigl\| Q_n-Q_n\wedge \sigma_n\,\rho_n^{-1}P_n \bigr\| = o(\sigma_n),
\]
then $\sigma_n^{-1}Q_n\contig \rho_n^{-1}P_n$. We leave the formulation
of other sufficient conditions to the reader. Note that inequality
(\ref{eq:rcexp}) in example~\ref{ex:KLclose} implies that
$b_n^{-1} P_0^n\contig a_n^{-1} P^n$, for any $a_n\leq\exp(-n\alpha^2)$ with
$\alpha^2>\ft12\ep^2$ and $b_n=\exp(-n(\alpha^2-\ft12\ep^2))$.
It is noted that this implies that $\phi_n(\samplen)\convas{Q_n}0$ for
any $\phi_n:\scrX_n\to[0,1]$ such that $P_n\phi_n(\samplen)=o(\rho_n)$ (more
generally, this holds whenever $\sum_n\sigma_n<\infty$, as a consequence
of the first Borel-Cantelli lemma).

\subsection{Remote contiguity for Bayesian limits}
\label{sub:rcbayes}

The relevant applications in the context of Bayesian limit theorems
concern remote contiguity of the sequence of true distributions
$P_{\tht_0,n}$ with respect to local prior predictive distributions
$P_n^{\Pi_n|B_n}$, where the sets $B_n\subset\Tht$ are such that,
\begin{equation}
  \label{eq:ctglpp}
  P_{\tht_0,n} \ctg a_n^{-1}P_n^{\Pi_n|B_n},
\end{equation}
for some rate $a_n\downarrow0$.

In the case of \iid\ data, Barron \cite{Barron88} introduces strong
and weak notions of \emph{merging} of $P_{\tht_0,n}$ with
(non-local) prior predictive distributions $P_n^{\Pi}$.
The weak version imposes condition~{\it (ii)} of lemma~\ref{lem:rcfirstlemma}
for all exponential rates simultaneously. Strong merging (or \emph{matching}
\cite{Barron86}) coincides with Schwartz's almost-sure limit, while
weak matching is viewed as a limit in probability. 

By contrast, if we have a specific rate $a_n$ in mind, the relevant
mode of convergence is Prohorov's weak convergence: according to
lemma~\ref{lem:rcfirstlemma}-{\it (v)},
(\ref{eq:ctglpp}) holds if inverse likelihood ratios $Z_n$ have
a weak limit $Z$ when re-scaled by $a_n$,
\[
  Z_n=(dP_n^{\Pi_n|B_n}/dP_{\tht_0,n})^{-1}(X^n),\quad
  a_n\,Z_n\convweak{P_{\tht_0,n}}Z.
\]
To better understand the counterexamples of section~\ref{sec:app},
notice the high sensitivity of this criterion to the existence of
subsets of the sample spaces assigned probability zero under some
model distributions, while the true probability is non-zero.
More generally, remote contiguity is sensitive to subsets $E_n$
assigned fast decreasing probabilities under local prior predictive
distributions $P_n^{\Pi_n|B_n}(E_n)$, while 
the probabilities $P_{\tht_0,n}(E_n)$ remain high, which is what
definition~\ref{def:remctg} expresses. The rate $a_n\downarrow0$
helps to control the likelihood ratio (compare to the unscaled limits
of likelihood ratios that play a central role in the theory of convergence
of experiments \cite{LeCam86}), conceivably enough to force
uniform tightness in many non-parametric situations.

But condition (\ref{eq:ctglpp}) can also be written out, for example to the
requirement that for some constant $\delta>0$,
\[
  P_{\tht_0,n}\Bigl( \,\int
    \frac{dP_{\tht,n}}{dP_{\tht_0,n}}(\samplen)\,d\Pi_n(\tht|B_n)
    < \delta\,a_n\Bigr)\to0,
\]
with the help of lemma~\ref{lem:rcfirstlemma}-{\it(ii)}.
\begin{example}
\label{ex:domains}
Consider again the model of example~\ref{ex:noKLpriors}. 
In example~\ref{ex:rcnoKLpriors}, it is shown that if the prior $\Pi$
for $\tht\in\RR$ has a continuous and strictly positive Lebesgue density
and we choose $B_n=[\tht_0,\tht_0+1/n]$, then for every $\delta>0$ and
all $a_n\downarrow0$,
\[
  P_{\tht_0}^n\biggl( \int
    \frac{dP_{\tht,n}}{dP_{\tht_0,n}}(\samplen)\,d\Pi(\tht|B_n)
    < \delta\,a_n \biggr)
  \leq
  P_{\tht_0}^n\bigl( \,n(X_{(1)}-\tht_0) < 2 \delta\,a_n \,\bigr),
\]
for large enough $n\geq1$, and the \rhs\ goes to zero for any $a_n$
because the random variables $n(X_{(1)}-\tht_0)$ have a non-degenerate,
positive weak limit under $P_{\tht_0}^n$ as $n\to\infty$. Conclude
that with these choices for $\Pi$ and $B_n$, (\ref{eq:ctglpp})
holds, for any $a_n$. \closebox
\end{example}

The following proposition should be viewed in light of
\cite{LeCam88}, which considers properties like contiguity,
convergence of experiments and local asymptotic normality in situations
of statistical information loss. In this case, we are interested in
(remote) contiguity of the probability measures that arise as marginals
for the data $\samplen$ when information concerning the (Bayesian random)
parameter $\tht$ is unavailable.
\begin{proposition}
\label{prop:utfamily}
Let $\tht_0\in\Tht$ and a prior $\Pi:\scrG\to[0,1]$ be given.
Let $B$ be a measurable subset of $\Tht$ such that $\Pi(B)>0$.
Assume that for some $a_n\downarrow0$, the family,
\[
  \biggl\{ a_n\Bigl(\frac{dP_{\tht,n}}{dP_{\tht_0,n}}\Bigr)^{-1}(\samplen):
    \tht\in B, n\geq1 \biggr\},
\]
is uniformly tight under $P_{\tht_0,n}$. Then
$P_{\tht_0,n}\ctg a_n^{-1} P_{n}^{\Pi|B}$.
\end{proposition}
Other sufficient conditions from lemma~\ref{lem:rcfirstlemma} may replace
the uniform tightness condition. When the prior $\Pi$ and subset $B$ are
$n$-dependent, application of lemma~\ref{lem:rcfirstlemma} requires more.
(See, for instance, example~\ref{ex:rcLAN} and lemma~\ref{lem:rcLAN}, where
local asymptotic normality is used to prove (\ref{eq:ctglpp}).)

To re-establish contact with the notion of merging, note the following.
If remote contiguity of the type (\ref{eq:ctglpp}) can be achieved
for a sequence of subsets $(B_n)$, then it also holds for any sequence
of sets (\eg\ all equal to $\Tht$, in Barron's case) that contain the
$B_n$ but at a rate that differs proportionally to the fraction of prior
masses.
\begin{lemma}
\label{lem:rcsubset}
For all $n\geq1$, let $B_n\subset\Theta$ be such that $\Pi_n(B_n)>0$
and $C_n$ such that $B_n\subset C_n$ with
$c_n=\Pi_n(B_n)/\Pi_n(C_n)\downarrow0$,
then,
\[
  P_n^{\Pi_n|B_n} \ctg c_n^{-1}\,P_n^{\Pi_n|C_n}.
\]
Also, if for some sequence $(P_n)$, $P_n\ctg a_n^{-1}\,P_n^{\Pi_n|B_n}$
then $P_n\ctg a_n^{-1}c_n^{-1}\, P_n^{\Pi_n|C_n}$.
\end{lemma}
So when considering possible choices for the sequence $(B_n)$,
smaller choices lead to slower rates $a_n$, rendering (\ref{eq:defrc})
applicable to more sequences of test functions. This advantage is to
be balanced against later requirements that $\Pi_n(B_n)$ may not
decrease too fast.

\section{Posterior concentration}
\label{sec:four}

In this section new frequentist theorems are formulated involving
the convergence of posterior distributions. First we give a basic
proof for posterior consistency assuming existence of
suitable test sequences and remote contiguity of true distributions
$(P_{\tht_0,n})$ with respect to local prior
predictive distributions. Then it is not difficult to extend the proof
to the case of posterior rates of convergence in metric topologies.
With the same methodology it is possible to address questions
in Bayesian hypothesis testing and model selection: if a Bayesian
test to distinguish between two hypotheses exists and remote
contiguity applies, frequentist consistency of the Bayes Factor can be
guaranteed.
We conclude with a theorem that uses remote contiguity to describe a
general relation that exists between credible sets and confidence sets,
provided the prior induces remotely-contiguous local prior predictive
distributions.

\subsection{Consistent posteriors}
\label{sub:cons}

First, we consider
posterior consistency generalising Schwartz's theorem to sequentially
observed (non-\iid) data, non-dominated models and priors or
parameter spaces that may depend on the sample size. For an
early but very complete overview
of literature and developments in posterior consistency, see
\cite{Ghosal99}.
\begin{definition}
\label{def:consistency}
The posteriors $\Pi(\,\cdot\,|\samplen)$
are \emph{consistent at $\tht\in\Tht$} if for every neighbourhood
$U$ of $\tht$,
\begin{equation}
  \label{eq:cons}
  \Pi(U|\samplen)\conv{P_{\tht,n}}1.
\end{equation}
The posteriors are said to be \emph{consistent} if this holds for
all $\tht\in\Tht$. We say that
the posterior is \emph{almost-surely consistent} if convergence
occurs almost-surely with respect to some coupling for the sequence
$(P_{\tht_0,n})$.
\end{definition}
Equivalently, posterior consistency can be characterized in terms
of posterior expectations of bounded and continuous functions (see
proposition~\ref{prop:prokhorov}).
\begin{theorem}
\label{thm:consistency}
Assume that for all $n\geq1$, the data $\samplen\sim P_{\tht_0,n}$ for
some $\tht_0\in\Tht$. Fix a prior $\Pi:\scrG\to[0,1]$ and assume that
for given $B,V\in\scrG$ with $\Pi(B)>0$ and $a_n\downarrow0$,
\begin{itemize}[noitemsep,nolistsep]
  \item[(i)] there exist Bayesian tests $\phi_n$ for $B$ versus $V$,
    \begin{equation}
      \label{eq:bayesiantestingpower}
      \int_{B} P_{\tht,n}\phi_n\,d\Pi(\tht)
        + \int_{V} P_{\tht',n}(1-\phi_n)\,d\Pi(\tht') = o(a_n),
    \end{equation}
  \item[(ii)] the sequence $P_{\tht_0,n}$ satisfies
    $P_{\tht_0,n} \, \ctg \, a_n^{-1}\, P_n^{\Pi|B}$.
\end{itemize}
Then $\Pi(V|\samplen)\convprob{P_{\tht_0,n}}0$.
\end{theorem}
These conditions are to be interpreted as follows:
theorem~\ref{thm:testconsequi} lends condition~{\it(i)} a distinctly
Bayesian interpretation: it requires a Bayesian test to set $V$ apart
from $B$ with testing power $a_n$. Lemma~\ref{lem:testineq} translates
this into the (still Bayesian) statement that the posteriors for $V$
go to zero in $P_n^{\Pi|B}$-expectation. Condition~{\it(ii)} is there
to promote this Bayesian point to a frequentist one
through (\ref{eq:defrc}). To present this from another perspective:
condition~{\it(ii)} ensures that the $P_n^{\Pi|B}$ cannot be
tested versus $P_{\tht_0,n}$ at power $a_n$, so the posterior for $V$
go to zero in $P_{\tht_0,n}$-expectation as well (otherwise a sequence
$\phi_n(\samplen)\propto \Pi(V|\samplen)$ would
constitute such a test).

To illustrate theorem~\ref{thm:consistency} and its conditions
Freedman's counterexamples are considered in detail in
example~\ref{ex:contfreedman65}.

A proof of a theorem very close to Schwartz's theorem is now possible.
Consider condition {\it (i)} of theorem~\ref{thm:schwartz}: a well-known
argument based on Hoeffding's inequality guarantees the existence of
a uniform test sequence \emph{of exponential power} whenever a
uniform test sequence test sequence exists, so Schwartz
equivalently assumes that there exists a $D>0$ such that,
\[
  P_0^n\phi_n +
  \sup_{ Q\in \scrP\setminus U} Q^n(1-\phi_n) =o(e^{-nD}).
\]
We vary slightly and assume the existence of a Bayesian test sequence
of exponential power. In the following theorem, let $\scrP$ denote a
Hausdorff space of single-observation distributions on $(\scrX,\scrB)$
with Borel prior $\Pi$.
\begin{corollary}
\label{cor:schwartz}
For all $n\geq1$, let $(X_1, X_2, \ldots, X_n)\sim P_0^n$ for some
$P_0\in\scrP$. Let $U$ denote an open neighbourhood of $P_0$ and define
$K(\ep)=\{P\in\scrP:-P_0\log(dP/dP_0)<\ep^2\}$. If,
\begin{itemize}[noitemsep,nolistsep]
\item[(i)] there exist $\ep>0$, $D>0$ and a sequence of measurable
  $\psi_n:\scrX^n\to[0,1]$, such that,  
  \[
    \int_{K(\ep)} P^n\psi_n\,d\Pi(P)
    + \int_{\scrP\setminus U} Q^n(1-\psi_n)\,d\Pi(Q)
      =o(e^{-nD}),
  \]
\item[(ii)] and $\Pi(K(\ep))>0$ for all $\ep>0$,
\end{itemize}
then $\Pi(U|\samplen)\convas{P_0}1$.
\end{corollary}
An instance of the application of corollary~\ref{cor:schwartz} is given
in example~\ref{ex:schwartz}.
Example~\ref{ex:finiteX} demonstrates posterior consistency in total
variation for \iid\ data from a finite sample space, for priors of
full support. Extending this, example~\ref{ex:tailfree} concerns
consistency of posteriors for priors that have Freedman's tailfreeness
property \cite{Freedman65}, like the Dirichlet process prior.
Also interesting in this respect is the Neyman-Scott paradox, a
classic example of inconsistency for the ML estimator, discussed
in Bayesian context in \cite{Bayarri04}: whether the posterior is
(in)consistent depends on the prior. The Jeffreys
prior follows the ML estimate while the reference prior avoids
the Neyman-Scott inconsistency. Another question in a sequence
model arises when we analyse FDR-like posterior
consistency for a sequence vector that is assumed to be sparse
(see example~\ref{ex:sparsenormalmeans}).

\subsection{Rates of posterior concentration}
\label{sub:rate}

A significant extension to the theory on posterior convergence
is formed by results concerning posterior convergence in metric spaces
\emph{at a rate}.
Minimax rates of convergence for (estimators based on) posterior
distributions were considered more or less simultaneously in
Ghosal-Ghosh-van~der~Vaart \cite{Ghosal00} and
Shen-Wasserman \cite{Shen01}. Both propose an extension of
Schwartz's theorem to posterior rates of convergence
\cite{Ghosal00,Shen01} and apply Barron's sieve idea with a
well-known entropy argument \cite{Birge83,Birge84} to a shrinking sequence of
Hellinger neighbourhoods and employs a more specific, rate-related
version of the Kullback-Leibler condition (\ref{eq:KLprior}) for
the prior. Both appear to be inspired by contemporary results
regarding Hellinger rates of convergence for sieve MLE's, as well as
on Barron-Schervish-Wasserman \cite{Barron99}, which concerns
posterior consistency based on
controlled bracketing entropy for a sieve, up to subsets of negligible
prior mass, following ideas that were first laid down in
\cite{Barron88}. It is remarked already in \cite{Barron99} that
their main theorem is easily re-formulated as a rate-of-convergence
theorem, with reference to \cite{Shen01}. More recently, Walker,
Lijoi and Pr\"unster \cite{Walker07} have added to these
considerations with a theorem for Hellinger rates of posterior
concentration in models that are separable for the Hellinger metric,
with a central condition that calls for summability of square-roots
of prior masses of covers of the model by Hellinger balls, based
on analogous consistency results in Walker \cite{Walker04}.
More recent is \cite{Kleijn16}, which shows that alternatives
for the priors of \cite{Ghosal00,Shen01} exist.

\begin{theorem}
\label{thm:rates}
Assume that for all $n\geq1$, the data $\samplen\sim P_{\tht_0,n}$ for
some $\tht_0\in\Tht$. Fix priors $\Pi_n:\scrG\to[0,1]$ and assume that
for given $B_n,V_n\in\scrG$ with $\Pi_n(B_n)>0$ and $a_n,b_n\downarrow0$
such that $a_n=o(b_n)$,
\begin{itemize}
\item[(i)] there are Bayesian tests $\phi_n:\scrX_n\to[0,1]$ such that,
\begin{equation}
  \label{eq:bayesiantestingpowerrates}
  \int_{B_n} P_{\tht,n}\phi_n\,d\Pi_n(\tht)
    + \int_{V_n} P_{\tht,n}(1-\phi_n)\,d\Pi_n(\tht) = o(a_n),
\end{equation}
\item[(ii)] The prior mass of $B_n$ is lower-bounded, $\Pi_n(B_n)\geq b_n$,
\item[(iii)] The sequence $P_{\tht_0,n}$ satisfies
$P_{\tht_0,n} \, \ctg \, b_na_n^{-1}\, P_n^{\Pi_n|B_n}$.
\end{itemize}
Then $\Pi(V_n|\samplen)\convprob{P_{\tht_0,n}}0$.\closebox
\end{theorem}

\begin{example}
To apply theorem~\ref{thm:rates}, consider again the situation
of a uniform distribution with an unknown location, as in\
examples~\ref{ex:noKLpriors} and~\ref{ex:domains}. Taking $V_n$
equal to $\{\tht:\tht-\tht_0>\ep_n\}$
$\{\tht:\tht_0-\tht>\ep_n\}$ respectively, with $\ep_n=M_n/n$
for some $M_n\to\infty$, suitable
test sequences are constructed in example~\ref{ex:testnoKLpriors},
and in combination with example~\ref{ex:domains},
lead to the conclusion that with a prior $\Pi$ for $\tht$ that
has a continuous and strictly positive Lebesgue density, the
posterior is consistent at (any $\ep_n$ slower than) rate $1/n$.
\end{example}

\begin{example}
\label{ex:GGV}
Let us briefly review the conditions of \cite{Barron99,Ghosal00,Shen01}
in light of theorem~\ref{thm:rates}: let $\ep_n\downarrow0$ denote the
Hellinger rate of convergence we have in mind, let $M>1$ be some constant
and define,
\[
  \begin{split}
  V_n&=\{P\in\scrP:H(P,P_0)\geq M\ep_n\},\\
  B_n&=\{P\in\scrP:-P_0\log{dP/dP_0}<\ep_n^2,
    \,P_0\log^2{dP/dP_0}<\ep_n^2\}.
  \end{split}
\]
Theorems for posterior convergence at a rate propose a sieve of
submodels satisfying entropy conditions like those of
\cite{Birge83,Birge84,LeCam86} and a negligibility condition for
prior mass outside the sieve \cite{Barron88}, based on the
minimax Hellinger rate of convergence $\ep_n\downarrow0$. Together,
they guarantee the existence of
Bayesian tests for Hellinger balls of radius $\ep_n$ versus
complements of Hellinger balls of radius $M\ep_n$ of power
$\exp(-DM^2\,n\ep_n^2)$ for some $D>0$ (see example~\ref{ex:entropy}).
Note that $B_n$ is contained in the Hellinger ball of radius
$\ep_n$ around $P_0$, so (\ref{eq:bayesiantestingpowerrates})
holds. New in \cite{Ghosal00,Shen01} is the condition for the
priors $\Pi_n$,
\begin{equation}
  \label{eq:GGV}
  \Pi_n(B_n)
  \geq e^{-Cn\ep_n^2},
\end{equation}
for some $C>0$. With the help of lemmas~\ref{lem:ctgGGV}
and~\ref{lem:rcfirstlemma}-{\it(ii)}, we conclude that,
\begin{equation}
  \label{eq:GGVrc}
  P_0^n\ctg e^{cn\ep_n^2} P_n^{\Pi|B_n},
\end{equation}
for any $c>1$. If we choose $M$ such that $DM^2-C>1$,
theorem~\ref{thm:rates}
proves that $\Pi(V_n|\samplen)\convprob{P_0}0$, \ie\ the
posterior is Hellinger consistent at rate $\ep_n$.
\end{example}
Certain (simple, parametric) models do not allow the definition
of priors that satisfy (\ref{eq:GGV}), and alternative less
restrictive choices for the sets $B_n$ are possible under
mild conditions on the model \cite{Kleijn16}.

\subsection{Consistent hypothesis testing with Bayes factors}
\label{sub:fact}

The Neyman-Pearson paradigm notwithstanding, hypothesis testing
and classification concern the same fundamental statistical question,
to find a procedure to choose one subset from a given partition
of the parameter space as the most likely to contain the parameter
value of the distribution that has generated the data observed.
Asymptotically one wonders whether choices following
such a procedure focus on the correct subset with probability
growing to one.

From a somewhat shifted perspective, we argue as follows: no
statistician can be certain of the validity of specifics in his
model choice and therefore always runs the risk of biasing his
analysis from the outset. Non-parametric approaches alleviate
his concern but imply greater uncertainty within the model, leaving
the statistician with the desire to select the correct (sub)model
on the basis of the data before embarking upon the
statistical analysis proper (for a recent overview, see
\cite{Taylor15}). The issue also
makes an appearance in asymptotic context, where over-parametrized
models leave room for inconsistency of estimators, requiring
regularization \cite{Birge97,Birge01,Buhlman11}.

\emph{Model selection} describes all statistical methods that
attempt to determine from the data which model to use. (Take for
example sparse variable selection, where one projects out the
majority of covariates prior to actual estimation, and the
model-selection question is which projection is optimal.)
Methods for model selection range from simple rules-of-thumb,
to cross-validation and penalization of the likelihood function.
Here we propose to conduct the frequentist analysis with the help
of a posterior: when faced with a (dichotomous) model choice,
we let the so-called Bayes factor formulate our
preference. For an analysis of hypothesis testing that compares
Bayesian and frequentist views, see \cite{Bayarri04}. An
objective Bayesian perspective on model selection
is provided in \cite{Wasserman06}.

\begin{definition}
For all $n\geq1$, let the model be parametrized by maps
$\tht\mapsto P_{\tht,n}$ on a parameter space $(\Tht,\scrG)$
with priors $\Pi_n:\scrG\rightarrow[0,1]$. Consider disjoint,
measurable $B,V\subset\Tht$. For given $n\geq1$, we say that the
\emph{Bayes factor for testing $B$ versus $V$},
\[
    F_n = \frac{\Pi(B|\samplen)}{\Pi(V|\samplen)}
    \frac{\Pi_n(V)}{\Pi_n(B)},
\]
is consistent for testing $B$ versus $V$, if for all $\tht\in V$,
$F_n\convprob{P_{\tht,n}}0$ and for all $\tht\in B$,
$F_n^{-1}\convprob{P_{\tht,n}}0$.
\end{definition}
Let us first consider this from a purely Bayesian perspective: for
fixed prior $\Pi$ and $\iid$ data, theorem~\ref{thm:testconsequi}
says that the posterior gives rise to consistent Bayes factors for
$B$ versus $V$ in a Bayesian (that is, $\Pi$-almost-sure) way, iff
a Bayesian test sequence for $B$ versus $V$ exists. If the parameter
space $\Tht$ is Polish and the maps $\tht\mapsto P_\tht(A)$ are
Borel measurable for all $A\in\scrB$, proposition~\ref{prop:msbtest}
says that any Borel set $V$ is Bayesian testable versus $\Tht\setminus V$,
so in Polish models for \iid\ data, model selection with Bayes factors is
$\Pi$-almost-surely consistent for all Borel measurable $V\subset\Tht$.

The frequentist requires strictly more, however, so we employ remote
contiguity again to bridge the gap with the Bayesian formulation.
\begin{theorem}
\label{thm:bayesfactor}
For all $n\geq1$, let the model be parametrized by maps
$\tht\mapsto P_{\tht,n}$ on a parameter space with $(\Tht,\scrG)$
with priors $\Pi_n:\scrG\rightarrow[0,1]$. Consider disjoint,
measurable $B,V\subset\Tht$ with $\Pi_n(B),\Pi_n(V)>0$ such that,
\begin{itemize}[noitemsep,nolistsep]
  \item[(i)] There exist Bayesian tests
    for $B$ versus $V$ of power $a_n\downarrow0$,
    \[
      \int_{B} P^n\phi_n\,d\Pi_n(P)
        + \int_{V} Q^n(1-\phi_n)\,d\Pi_n(Q) = o(a_n),
    \]
  \item[(ii)] For every $\tht\in B$, $P_{\tht,n} \ctg a_n^{-1}P_n^{\Pi_n|B}$,
    and for every $\tht\in V$, $P_{\tht,n} \ctg a_n^{-1}P_n^{\Pi_n|V}$.
\end{itemize}
Then the Bayes factor for $B$ versus $V$ is consistent.
\end{theorem}
Note that the second condition of theorem~\ref{thm:bayesfactor} can
be replaced by a local condition: if, for every $\tht\in B$,
there exists a sequence $B_n(\tht)\subset B$ such that
$\Pi_n(B_n(\tht))\geq b_n$ and $P_{\tht,n} \ctg a_n^{-1}b_nP_n^{\Pi_n|B_n}$,
then $P_{\tht,n} \ctg a_n^{-1}P_n^{\Pi_n|B}$ (as a consequence
of lemma~\ref{lem:rcsubset} with $C_n=B$).

In example~\ref{ex:gofmarkov}, we use theorem~\ref{thm:bayesfactor}
to prove the consistency of the Bayes factor for a goodness-of-fit
test for the equilibrium distribution of an stationary ergodic
Markov chain, based on large-length random-walk data, with prior
and posterior defined on the space of Markov transition matrices.


\subsection{Confidence sets from credible sets}
\label{sub:conf}

The Bernstein-von~Mises theorem \cite{LeCam90} asserts that the
posterior for a smooth, finite-dimensional parameter converges
in total variation to a normal distribution centred on an efficient
estimate with the inverse Fisher information as its covariance,
if the prior has full support. The methodological implication is
that Bayesian credible sets derived from such a posterior can be
reinterpreted as asymptotically efficient confidence sets.
This parametric fact begs for the exploration of possible
non-parametric extensions but Freedman discourages us
\cite{Freedman99} with counterexamples (see also \cite{Cox93})
and concludes that: {\it ``The sad lesson for inference is this.
If frequentist coverage probabilities are wanted in an
infinite-dimensional problem, then frequentist coverage probabilities
must be computed.''}

In recent years, much effort has gone into calculations that
address the question whether non-parametric credible sets can play
the role of confidence sets nonetheless. The focus lies on well-controlled
examples in which both model and prior are Gaussian so that the
posterior is conjugate and analyse posterior expectation and variance to
determine whether credible metric balls have asymptotic frequentist
coverage (for examples, see Szab\'o, van~der~Vaart and
van~Zanten \cite{Szabo15} and references therein).
Below, we change the question slightly and do not seek
to justify the use of credible sets as confidence sets; from
the present perspective it appears
more natural to ask in which particular fashion a credible set is to
be transformed in order to guarantee the transform is a
confidence set, at least in the large-sample limit.

In previous subsections, we have applied remote contiguity after the
concentration inequality to control the $P_{\tht_0,n}$-expectation of the
posterior probability for the alternative $V$ through its
$P_n^{\Pi|B_n}$-expectation. In the discussion of the coverage of credible
sets that follows, remote contiguity is applied to control the
$P_{\tht_0,n}$-probability that $\tht_0$ falls outside the prospective
confidence set through its $P_n^{\Pi|B_n}$-probability. The theorem below
then follows from an application of Bayes's rule (\ref{eq:disintegration}).
Credible levels provide the sequence $a_n$.

\begin{definition}
\label{def:cred}
Let $(\Tht,\scrG)$ with prior $\Pi$, denote the sequence of posteriors by
$\Pi(\cdot|\cdot):\scrG\times\scrX_n\to[0,1]$. Let $\scrD$ denote a
collection of measurable subsets of $\Tht$. 
A \emph{sequence of credible sets} $(D_n)$ of
\emph{credible levels} $1-a_n$ (where $0\leq a_n\leq1$, $a_n\downarrow0$)
is a sequence of set-valued maps $D_n:\scrX_n\to\scrD$ such that
$\Pi(\Tht\setminus D_n(x)|x)=o(a_n)$ for $P_n^{\Pi_n}$-almost-all
$x\in\scrX_n$.
\end{definition}
\begin{definition}
\label{def:confset}
For $0\leq a\leq1$, a set-valued map $x\mapsto C(x)$ defined on
$\scrX$ such that, for all $\tht\in\Tht$,
$P_{\tht}(\tht\not\in C(X))\leq a$, is called a
confidence set of level $1-a$. If the levels $1-a_n$ of a sequence
of confidence sets $C_n(X^n)$ go to $1$ as $n\to\infty$, 
the $C_n(\samplen)$ are said to be asymptotically consistent.
\end{definition}
\begin{definition}
\label{def:confcred}
Let $D$ be a (credible) set in $\Tht$ and let
$B=\{B(\tht):\tht\in\Tht\}$ denote a collection of model
subsets such that $\tht\in B(\tht)$ for all $\tht\in\Tht$.
A model subset $C'$ is said to be (a confidence set)
associated with $D$ under $B$, if for all $\tht\in\Tht\setminus C'$,
$B(\tht)\cap D=\emptyset$.
The intersection $C$ of all $C'$ like above equals
$\{\tht\in\Tht:B(\tht)\cap D\neq\emptyset\}$ and is
called the minimal (confidence) set associated with $D$ under $B$
(see Fig~\ref{fig:assoc}).
\end{definition}
\begin{figure}[bht]
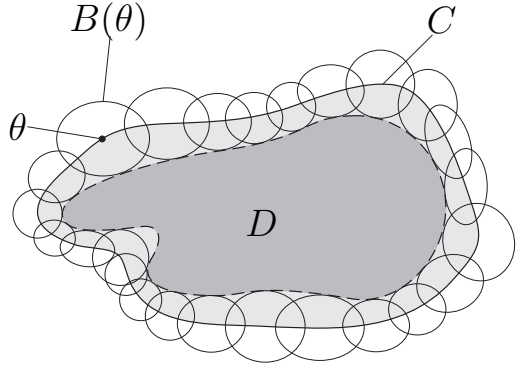

  \begin{center}
    \parbox{8.7cm}{
      \begin{lpic}{assoc-2(0.3)}
        \lbl[t]{43,133;{\Large $\tht$}}
        \lbl[t]{83,182;{\Large $B(\tht)$}}
        \lbl[t]{150,90;{\Large $D$}}
        \lbl[t]{229,180;{\Large $C$}}
      \end{lpic}
      \caption{\label{fig:assoc}The relation between a
      credible set $D$ and its associated (minimal)
      confidence set $C$ under $B$
      in Venn diagrams: the extra points $\tht$ in the associated
      confidence set $C$ not included in the credible set
      $D$ are characterized by non-empty intersection
      $B(\tht)\cap D\neq\emptyset$.}
    }
  \end{center}
\end{figure}
Example~\ref{ex:uniform} makes this construction explicit in uniform
spaces and specializes to metric context.
\begin{theorem}
\label{thm:coverage}
Let $\tht_0\in\Tht$ and $0\leq a_n\leq1$, $b_n>0$ such that $a_n=o(b_n)$
be given. Choose priors $\Pi_n$ and let $D_n$ denote level-$(1-a_n)$
credible sets. Furthermore, for all $\tht\in\Tht$, let
$B_n=\{B_n(\tht)\in\scrG:\tht\in\Tht\}$ denote a sequence such that,
\begin{itemize}
\item[(i)] $\Pi_n(B_n(\tht_0))\geq b_n$,
\item[(ii)] 
  $P_{\tht_0,n}\ctg b_na_n^{-1}\, P_n^{\Pi_n|B_n(\tht_0)}$.
\end{itemize}
Then any confidence sets $C_n$ associated with the credible
sets $D_n$ under $B_n$ are asymptotically consistent, \ie\ for
all $\tht_0\in\Tht$,
\begin{equation}
  \label{eq:coverage}
  P_{\tht_0,n}\bigl(\,\tht_0\in C_n(\samplen)\,\bigr)\to 1.
\end{equation}
\end{theorem}
This refutes Freedman's lesson, showing that the asymptotic
identification of credible sets and confidence sets in smooth
parametric models (the main inferential implication of the
Bernstein-von~Mises theorem) generalises to the above form of
asymptotic congruence in non-parametric models. The fact
that this statement holds in full
generality implies very practical ways to obtain confidence
sets from posteriors, calculated, simulated or approximated.
A second remark concerns the
confidence levels of associated confidence sets. In order for the
assertion of theorem~\ref{thm:coverage} to be specific regarding the
confidence level (rather than just resulting in asymptotic coverage),
we re-write the last condition of theorem~\ref{thm:coverage} as
follows,
\begin{itemize}
\item[{\it (ii')}] $c_n^{-1}P_{\tht_0,n}
  \ctg b_na_n^{-1}\,P_n^{\Pi_n|B_n(\tht_0)}$,
\end{itemize}
so that the last step in the proof of theorem~\ref{thm:coverage}
is more specific; particularly, assertion~(\ref{eq:coverage}) becomes,
\[
  P_{\tht_0,n}\bigl(\,\tht\in D_n(\samplen)\,\bigr)=o(c_n),
\]
\ie\ the confidence level of the sets $D_n(\samplen)$ is $1-Kc_n$
asymptotically (for some constant $K>0$ and large enough $n$).

The following corollary that specializes to the \iid\ situation
is immediate (see example~\ref{ex:confballs}). Let
$\scrP$ denote a model of single-observation
distributions, endowed with the Hellinger or total-variational
topology.
\begin{corollary}
\label{cor:hellconf}
For $n\geq1$ assume that $(X_1, X_2, \ldots, X_n)\in\scrX^n\sim P_0^n$
for some $P_0\in\scrP$. Let $\Pi_n$ denote Borel priors on $\scrP$,
with constant $C>0$ and rate sequence $\ep_n\downarrow0$ such that
(\ref{eq:GGV}) is satisfied. Denote by $D_n$
credible sets of level $1-\exp(-C'n\ep_n^2)$, for some $C'>C$. Then
the confidence sets $C_n$ associated with $D_n$ under
radius-$\ep_n$ Hellinger-enlargement are asymptotically consistent.
\end{corollary}
Note that in the above corollary,
\[
  \text{diam}_H(C_n(\samplen)) = \text{diam}_H(D_n(\samplen)) + 2\ep_n,
\]
$P_0^n$-almost surely. If, in addition to the conditions in the above
corollary,
tests satisfying (\ref{eq:bayesiantestingpowerrates}) with
$a_n=\exp(-C'n\ep_n^2)$ exist, the posterior is consistent at rate
$\ep_n$ and sets $D_n(\samplen)$ have diameters decreasing as $\ep_n$,
\cf\ theorem~\ref{thm:rates}. In the case $\ep_n$ is the
minimax rate of convergence for the problem, the confidence sets
$C_n(\samplen)$ attain rate-optimality \cite{Low97}.
Rate-adaptivity \cite{Hengartner95,Cai13,Szabo15} is not possible like
this because a definite, non-data-dependent choice for the
$B_n$ is required.

\section{Conclusions}
\label{sec:disc}

We list and discuss the main conclusions of this paper below.
\begin{itemize}
\item[] {\it Frequentist validity of Bayesian limits}\\
There exists a systematic way of taking Bayesian limits into frequentist
ones, if priors satisfy an extra condition relating true data
distributions to localized prior predictive distributions.
This extra condition generalises Schwartz's Kullback-Leibler
condition and amounts to a weakened form of contiguity,
termed \emph{remote contiguity}.
\end{itemize}
For example regarding consistency with \iid\ data, Doob shows
that a Bayesian form of posterior consistency holds
without any real conditions on the model. To the frequentist,
`holes' of potential inconsistency remain, in null-sets of the
prior. Remote contiguity `fills the holes' and elevates the
Bayesian form of consistency to the frequentist one.
Similarly, prior-almost-surely consistent tests are
promoted to frequentist consistent tests and Bayesian credible
sets are converted to frequentist confidence sets.
\begin{itemize}
\item[] {\it The nature of Bayesian test sequences}\\
The existence of a Bayesian test sequence is equivalent
to consistent posterior convergence in the Bayesian,
prior-almost-sure sense. In theorems above, a Bayesian test
sequence thus represents the Bayesian
limit for which we seek frequentist validity through remote
contiguity. Bayesian test sequences are more abundant
than the more familiar uniform test sequences.
Aside from prior mass requirements
arising from remote contiguity, \emph{the prior should assign
little weight where testing power is weak and much
where testing power is strong}, ideally.
\end{itemize}
Example~\ref{ex:gofmarkov} illustrates the influence of the
prior when constructing a test sequence. Aside from
the familiar lower bounds for prior mass that arise from
remote contiguity, existence of Bayesian tests
also poses upper bounds for prior mass.
\begin{itemize}
\item[] {\it Systematic analysis of complex models and datasets}\\ 
Although many examples have been studied on a case-by-case
basis in the literature, the systematic analysis of limiting
properties of posteriors in cases where the data is dependent,
or where the model, the parameter space and/or the prior
are sample-size dependent, requires generalisation of
Schwartz's theorem and its variations, which the formalism
presented here provides.
\end{itemize}
To elaborate, given the growing interest in the analysis of
dependent datasets gathered from networks (\eg\ by \emph{webcrawlers}
that random walk linked webpages), or from time-series/stochastic
processes (\eg\ financial data of the high-frequency type), or
in the form of high-dimensional or even functional data
(biological, financial, medical and meteorological fields
provide many examples), the development of new Bayesian methods
involving such aspects benefits from a simple, insightful, systematic
perspective to guide the search for suitable priors in concrete
examples.

To illustrate the last point, let us consider consistent
community detection in stochastic block models
\cite{Snijders97,Bickel09}. Bayesian methods have been developed
for consistent selection of the number of
communities \cite{Hayashi16}, for community detection with
a controlled error-rate with a growing number of communities
\cite{Choi12} and for consistent community detection
using empirical priors \cite{Suwan16}. A moment's thought
on the discrete nature of the community assignment
vector suggests a sequence of uniform priors, for which
remote contiguity (of $B_n=\{P_{0,n}\}$) is guaranteed
(at any rate) and prior mass lower bounded by
$b_n=K_n!K_n^{-n}$ (where $K_n$ is the number
of communities at `sample size' $n$). It
would be interesting to see under which conditions a
Bayesian test sequence of power $a_n=o(b_n)$ can be devised that
tests the true assignment vector versus all
alternatives (in the sparse regime \cite{Decelle11,Abbe12,Mossel16}).
Rather than apply a Chernoff bound like in
\cite{Choi12}, one would probably have to start from
the probabilistic \cite{Mossel16} or information-theoretic
\cite{Abbe12} analyses of respective algorithmic
solutions in the (very closely related)
planted bi-section model. If a suitably powerful test can be
shown to exist, theorem~\ref{thm:rates}
proves frequentist consistency of the posterior.

\begin{itemize}
  \item[] {\it Methodology for uncertainty quantification}\\ 
  Use of a prior that induces remote contiguity allows one
  to convert credible sets of a
  calculated, simulated or approximated posterior into asymptotically
  consistent confidence sets, in full generality. This extends the
  main inferential implication of the Bernstein-von~Mises theorem to
  non-parametric models without smoothness conditions.
\end{itemize}

The latter conclusion forms the most important and practically
useful aspect of this paper.


\appendix
\renewcommand{\theequation}{\Alph{section}.\arabic{equation}}


\section{Definitions and conventions}
\label{sec:defs}

Because we take the perspective of a frequentist using Bayesian
methods, we are obliged to demonstrate that Bayesian definitions
continue to make sense under the assumptions that the data $X$ is
distributed according to a true, underlying $P_0$.

\begin{remark}
\label{rem:conv}
We assume given for every $n\geq1$, a measurable (sample) space
$(\scrX_n,\scrB_n)$ and random sample $\samplen\in\scrX_n$,
with a model $\scrP_n$ of probability distributions
$P_n:\scrB_n\rightarrow[0,1]$. It is also assumed that there
exists an $n$-independent parameter space $\Tht$ with a Hausdorff,
completely regular topology $\scrT$ and associated Borel
$\sigma$-algebra $\scrG$,
and, for every $n\geq1$, a bijective model parametrization
$\Tht\rightarrow\scrP_n:\tht\mapsto P_{\tht,n}$ such
that for every $n\geq1$ and every $A\in\scrB_n$, the map
$\Tht\to[0,1]:\tht\mapsto P_{\tht,n}(A)$ is measurable.
Any prior $\Pi$ on $\Tht$ is assumed to be a Borel
probability measure $\Pi:\scrG\to[0,1]$ and can vary with
the sample-size $n$. (Note: in \iid\ setting, the parameter
space $\Tht$ is $\scrP_1$, $\tht$ is the single-observation
distribution $P$ and $\tht\mapsto P_{\tht,n}$ is
$P\mapsto P^n$.)
As frequentists, we assume that there exists a `true, underlying
distribution for the data; in this case, that means that for every
$n\geq1$, there exists a distribution $P_{0,n}$ from which the
$n$-th sample $\samplen$ is drawn.\closebox
\end{remark}

Often one assumes that the model is \emph{well-specified}: that
there exists a $\tht_0\in\Tht$ such that
$P_{0,n}=P_{\tht_0,n}$ for all $n\geq1$. We think
of $\Tht$ as a topological space because we want to discuss
estimation as a procedure of sequential, stochastic approximation of
and convergence to such a `true parameter value $\tht_0$. 
In theorem~\ref{thm:testconsequi} and definition~\ref{def:consistency}
we assume, in addition, that the observations $\samplen$ are
\emph{coupled}, \ie\ there exists a probability space $(\Omega,\scrF,P_0)$
and random variables $\samplen:\Omega\rightarrow\scrX_n$
such that $P_0((\samplen)^{-1}(A))=P_{0,n}(\samplen\in A)$ for all
$n\geq1$ and $A\in\scrB_n$.
\begin{definition}
Given $n,m\geq1$ and a prior probability measure $\Pi_n:\scrG\to[0,1]$,
define the \emph{$n$-th prior predictive distribution on $\scrX_m$} as
follows:
\begin{equation}
  \label{eq:priorpred}
  P_m^{\Pi_n}(A) = \int_\Tht P_{\tht,m}(A)\,d\Pi_n(\tht),
\end{equation}
for all $A\in\scrB_m$. If the prior is replaced by the posterior,
the above defines the \emph{$n$-th posterior predictive distribution
on $\scrX_m$},
\begin{equation}
  \label{eq:posteriorpred}
  P_m^{\Pi_n|\samplen}(A) = \int_\Tht P_{\tht,m}(A)\,d\Pi(\tht|\samplen),
\end{equation}
for all $A\in\scrB_m$. For any $B_n\in\scrG$ with $\Pi_n(B_n)>0$,
define also the \emph{$n$-th local prior predictive distribution on $\scrX_m$},
\begin{equation}
  \label{eq:localpriorpred}
  P_m^{\Pi_n|B_n}(A) = \frac{1}{\Pi_n(B_n)}\int_{B_n} P_{\tht,m}(A)\,d\Pi_n(\tht),
\end{equation}
as the predictive distribution on $\scrX_m$ that results from the prior
$\Pi_n$ when conditioned on $B_n$. If $m$ is not mentioned explicitly, it is
assumed equal to $n$. 
\end{definition}
The prior predictive distribution $P_n^{\Pi_n}$ is the marginal
distribution for $\samplen$ in the Bayesian perspective that
considers parameter and sample jointly
$(\tht,\samplen)\in\Tht\times\scrX_n$
as the random quantity of interest. 
\begin{definition}
\label{def:posterior}
Given $n\geq1$, 
a \emph{(version of) the posterior} is any map
$\Pi(\,\cdot\,|\samplen=\,\cdot\,):\scrG\times\scrX_n\rightarrow[0,1]$
such that,
\begin{itemize}[noitemsep,nolistsep]
\item[(i)] for $B\in\scrG$, the map
  $\scrX_n\rightarrow[0,1]:
    \realizationn\mapsto\Pi(B|\samplen=\realizationn)$ is
  $\scrB_n$-measurable,
\item[(ii)] for all $A\in\scrB_n$ and $V\in\scrG$,
  \begin{equation}
    \label{eq:disintegration}
    \int_A\Pi(V|\samplen)\,dP_n^{\Pi_n} = \int_V P_{\tht,n}(A)\,d\Pi_n(\tht).
  \end{equation}
\end{itemize}
\end{definition}
Bayes's Rule is expressed through equality (\ref{eq:disintegration})
and is sometimes referred to as a `disintegration' (of the joint
distribution of $(\tht,X^n)$). If the posterior is a Markov kernel, it is a
$P_n^{\Pi_n}$-almost-surely well-defined probability measure
on $(\Tht,\scrG)$. But it does not follow from the
definition above that a version of the posterior actually exists
as a regular conditional probability measure. Under mild extra
conditions, regularity of the posterior can be guaranteed: for
example, if sample space and parameter space are Polish, the posterior
is regular; if the model $\scrP_n$ is dominated (denote the density
of $P_{\tht,n}$ by $p_{\tht,n}$), the fraction of integrated likelihoods,
\begin{equation}
  \label{eq:posteriorfraction}
  \Pi(V|\samplen)= 
  {\displaystyle{\int_V p_{\tht,n}(\samplen)\,d\Pi_n(\tht)}} \biggm/
  {\displaystyle{\int_\Tht p_{\tht,n}(\samplen)\,d\Pi_n(\tht)}},
\end{equation}
for $V\in\scrG$, $n\geq1$ defines a regular version of the posterior
distribution. (Note also that there is no room in
definition (\ref{eq:disintegration}) for $\samplen$-dependence
of the prior, so `empirical Bayes' methods must be based on
data $Y^n$ independent of $\samplen$, \ie\ sample-splitting.)

\begin{remark}
\label{rem:conv2}
As a consequence of the frequentist assumption that $\samplen\sim P_{0,n}$
for all $n\geq1$, the $P_n^{\Pi_n}$-almost-sure definition
(\ref{eq:disintegration}) of the posterior $\Pi(V|\samplen)$
does not make sense automatically \cite{Freedman63,Kleijn16}: null-sets of
$P_n^{\Pi_n}$ on which the definition of $\Pi(\,\cdot\,|\samplen)$ is
ill-determined, may not be null-sets of $P_{0,n}$. To prevent this, we
impose the domination condition,
\begin{equation}
\label{eq:dompriorpred}
  P_{0,n} \ll P_{n}^{\Pi_n},
\end{equation}
for every $n\geq1$.\closebox
\end{remark}
To understand the reason for (\ref{eq:dompriorpred}) in a perhaps more
familiar way, consider a dominated model and assume that for certain $n$,
(\ref{eq:dompriorpred}) is {\it not} satisfied. Then, using (\ref{eq:priorpred}),
we find,
\[
  P_{0,n}\Bigl(\int p_{\tht,n}(\samplen)\,d\Pi_n(\tht)=0\Bigr)>0,
\]
so the denominator in
(\ref{eq:posteriorfraction}) evaluates to zero with
non-zero $P_{0,n}$-probability.

To get an idea of sufficient conditions for (\ref{eq:dompriorpred}),
it is noted in \cite{Kleijn16} that in the case
of \iid\ data where $P_{0,n}=P_0^n$ for some marginal
distribution $P_0$, $P_0^n \ll P_{n}^{\Pi}$ for all $n\geq1$, if $P_0$ lies
in the Hellinger- or Kullback-Leibler-support of the prior $\Pi$. For the
generalisation to the present setting we are more precise and weaken
the topology appropriately. 
\begin{definition}
For all $n\geq1$, let $F_n$ denote the class of
all bounded, $\scrB_n$-measurable $f:\scrX_n\to\RR$. The topology
$\scrT_n$ is the initial topology on $\scrP_n$ for
the functions $\{P\mapsto Pf:f\in F_n\}$. \closebox
\end{definition}
Finite intersections of sets $U_{f,\ep}=\{(P,Q)\in\scrP_n^2:|(P-Q)f|<\ep\}$
($f\in\scrF_n$, $0\leq f\leq1$ and $\ep>0$), form
a fundamental system of entourages for a uniformity $\scrU_n$
on $\scrP_n$. A fundamental system of neighbourhoods for the
associated topology $\scrT_n$ on $\scrP$ is formed by finite
intersections of sets of the form,
\[
  W_{P,f,\ep}=\{Q\in\scrP_n:|(P-Q)f|<\ep\},
\]
with $P\in\scrP_n$, $f\in\scrF_n$, $0\leq f\leq1$ and $\ep>0$.

If we model single-observation distributions $P\in\scrP$
for an \iid\ sample, the topology $\scrT_n$ on $\scrP_n=\scrP^n$
induces a topology on $\scrP$ (which we also denote by $\scrT_n$)
for each $n\geq1$. The union $\scrT_{\infty}=\cup_n\scrT_n$ is
an inverse-limit topology that allows formulation of
conditions for the existence of consistent estimates that are
not only sufficient, but also necessary \cite{LeCam60a}, offering
a precise perspective on what is estimable \emph{and what is not}
in \iid\ context. The associated strong topology is that
generated by total variation (or, equivalently, the Hellinger
metric).

For more on these topologies, the reader is referred to Strasser
(1985) \cite{Strasser85} and to Le~Cam (1986) \cite{LeCam86}. We
note explicitly the following fact, which is a direct consequence of
Hoeffding's inequality.
\begin{proposition}
\label{prop:weaktests}
{\it (Uniform $\scrT_n$-tests)}\\
Consider a model $\scrP$ of single-observation distributions $P$ for
\iid\ samples
$(X_1,X_2,\ldots,X_n)\sim P^n$, $(n\geq1)$. Let $m\geq1$,
$\ep>0$, $P_0\in\scrP$
and a measurable $f:\scrX^m\rightarrow[0,1]$ be given. Define
$B = \bigl\{ P\in\scrP\,:\,|(P^m-P_0^m)f|<\ep\bigr\}$, and
$V = \bigl\{ P\in\scrP\,:\,|(P^m-P_0^m)f|\geq2\ep\bigr\}$.
There exist a uniform test sequence $(\phi_n)$ such that,
\[
  \sup_{P\in B} P^n\phi_n\leq e^{-nD},\quad
  \sup_{Q\in V} Q^n(1-\phi_n) \leq e^{-nD},
\]
for some $D>0$.
\end{proposition}
\begin{proof}
The proof is an application of Hoeffding's inequality
for the sum $\sum_{i=1}^nf(X_i)$ and is left to the reader.
\end{proof}
The topologies $\scrT_n$ also play a role for
condition~(\ref{eq:dompriorpred}).
\begin{proposition}
\label{prop:dompriorpred}
Let $(\Pi_n)$ be Borel priors on the Hausdorff uniform spaces $(\scrP_n,
\scrT_n)$. For any $n\geq1$, if $P_{0,n}$ lies in the $\scrT_n$-support of
$\Pi_n$, then $P_{0,n} \ll P_{n}^{\Pi_n}$.
\end{proposition}
\begin{proof} 
Let $n\geq1$ be given. For any $A\in\scrB_n$ and any $U'\subset\Tht$
such that $\Pi_n(U')>0$,
\[
  P_{0,n}(A) \leq \int P_{\tht,n}(A)\,d\Pi_n(\tht|U')
    +\sup_{\tht\in U'} |P_{\tht,n}(A)-P_{0,n}(A)|.
\]
Let $A\in\scrB_n$ be a null-set of $P_n^{\Pi_n}$; since $\Pi_n(U')>0$,
$\int P_{\tht,n}(A)\,d\Pi_n(\tht|U')=0$. For some $\ep>0$, take
$U'$ equal to the $\scrT_n$-basis element
$\{\tht\in\Tht:|P_{\tht,n}(A)-P_{\tht_0,n}(A)|<\ep\}$ to conclude that
$P_{\tht_0,n}(A)<\ep$ for all $\ep>0$. 
\end{proof}
In many situations, priors are Borel for the Hellinger topology,
so it is useful to observe that the Hellinger support of $\Pi_n$ in
$\scrP_n$ is always contained in the $\scrT_n$-support. 

\subsubsection*{Notation and conventions}

\lhs\ and \rhs\ refer to left-
and right-hand sides respectively.
For given probability measures $P,Q$ on a measurable space
$(\Omega,\scrF)$, we define
the Radon-Nikodym derivative $dP/dQ:\Omega\to[0,\infty)$,
$P$-almost-surely, referring {\it only} to the $Q$-dominated
component of $P$, following \cite{LeCam86}. We also \emph{define}
$(dP/dQ)^{-1}:\Omega\to(0,\infty]:\omega\mapsto 1/(dP/dQ(\omega))$,
$Q$-almost-surely. Given a $\sigma$-finite measure $\mu$ that
dominates both $P$ and $Q$ (\eg\ $\mu=P+Q$), denote $dP/d\mu=q$
and $dQ/d\mu=p$. Then the measurable map
$p/q\,1\{q>0\}:\Omega\rightarrow[0,\infty)$ is a
$\mu$-almost-everywhere version of $dP/dQ$, and
$q/p\,1\{q>0\}:\Omega\rightarrow[0,\infty]$ of $(dP/dQ)^{-1}$. 
Define total-variational and Hellinger distances
by $\|P-Q\|=\sup_A|P(A)-Q(A)|$ and
$H(P,Q)^2=1/2\int(p^{1/2}-q^{1/2})^2\,d\mu$, respectively.
Given random variables $Z_n\sim P_n$, weak convergence to a random
variable $Z$ is denoted by $Z_n\convweak{P_n}Z$, convergence
in probability by $Z_n\convprob{P_n}Z$ and almost-sure convergence
(with coupling $P^\infty$) by $Z_n\convas{P^{\infty}}Z$.
The integral of a
real-valued, integrable random variable $X$ with respect to a
probability measure $P$ is denoted $PX$, while integrals over
the model with respect to priors and posteriors are always written
out in Leibniz's notation. For any subset $B$ of a topological
space, $\bar{B}$ denotes the closure, $\mathring{B}$ the interior
and $\partial B$ the boundary. Given $\ep>0$ and a metric space
$(\Tht,d)$, the covering number $N(\ep,\Tht,d)\in\NN\cup\{\infty\}$
is the minimal cardinal of a cover of $\Tht$ by $d$-balls of radius
$\ep$. Given real-valued random variables $X_1,\ldots,X_n$, the
first order statistic is $X_{(1)}=\min_{1\leq i\leq n} X_i$.
The Hellinger diameter of a
model subset $C$ is denoted $\text{diam}_H(C)$ and the Euclidean
norm of a vector $\tht\in\RR^n$ is denoted $\|\tht\|_{2,n}$.

\section{Applications and examples}
\label{sec:app}

In this section of the appendix examples and applications are collected.

\subsection{Inconsistent posteriors}
\label{sub:inconsistent}

Calculations that demonstrate instances of posterior inconsistency are
many (for (a far-from-exhaustive list of) examples, see \cite{Diaconis86a,
Diaconis86b, Cox93, Diaconis93, Diaconis98, Freedman83, Freedman99}). In
this subsection, we discuss early examples of posterior
inconsistency that illustrate the potential for problems clearly and without
distracting technicalities.
\begin{example}
\label{ex:freedman63}
{\it (Freedman (1963) \cite{Freedman63})}\\
Consider a sample $X_1,X_2,\ldots$ of random positive integers. Denote the
space of all probability distributions on $\NN$ by $\Lambda$ and assume
that the sample is \iid-$P_0$, for some $P_0\in\Lambda$. For any $P\in\Lambda$,
write $p(i)=P(\{X=i\})$ for all $i\geq1$. The total-variational
and weak topologies on $\Lambda$ are equivalent
(defined, $P\to Q$ if $p(i)\to q(i)$ for all $i\geq1$).
Let $Q\in\Lambda\setminus\{P_0\}$ be given. To arrive at a prior with $P_0$
in its support, leading to a posterior that concentrates on $Q$, we
consider sequences $(P_m)$ and $(Q_n)$ such that $Q_m\to Q$ and
$P_m\to P_0$ as $m\to\infty$. The prior $\Pi$ places masses $\alpha_m>0$
at $P_m$ and $\beta_m>0$ at $Q_m$ ($m\geq1$), so that $P_0$ lies in the
support of $\Pi$. A careful construction of the distributions $Q_m$ 
that involves $P_0$, guarantees that the posterior satisfies,
\[
  \frac{\Pi(\{Q_m\}|\samplen)}{\Pi(\{Q_{m+1}\}|\samplen)}
    \convas{P_0} 0,
\]
that is, posterior mass is shifted further
out into the tail as $n$ grows to infinity, forcing all posterior mass
that resides in $\{Q_m:m\geq1\}$ into arbitrarily small neighbourhoods
of $Q$. In a second step, the distributions $P_m$ and prior weights
$\alpha_m$ are chosen such that 
the likelihood at
$P_m$ grows large for high values of $m$ and small for lower values
as $n$ increases, so that the posterior mass in $\{P_m:m\geq1\}$
also accumulates in the tail. However, the prior weights $\alpha_m$
may be chosen to decrease very fast with $m$, in such a way that,
\[
  \frac{\Pi(\{P_m:m\geq1\}|\samplen)}
       {\Pi(\{Q_m:m\geq1\}|\samplen)} \convas{P_0} 0,
\]
thus forcing all posterior mass into $\{Q_m:m\geq1\}$ as $n$ grows.
Combination of the previous two displays 
leads to the conclusion that for every neighbourhood $U_Q$ of $Q$,
\[
  \Pi(U_Q|\samplen)\convas{P_0}1,
\]
so the posterior is inconsistent. Other choices
of the weights $\alpha_m$ that place more prior mass in the tail
\emph{do} lead to consistent posterior distributions.\closebox
\end{example}
Some objected to Freedman's counterexample, because knowledge
of $P_0$ is required to construct the prior that causes inconsistency.
So it was possible to argue that Freedman's counterexample
amounted to nothing more than a demonstration that unfortunate
circumstances could be created, probably not a fact of great
concern in any generic sense.
To strengthen Freedman's point one would need to construct a prior
of full support without explicit knowledge of $P_0$.
\begin{example}
\label{ex:freedman65}
{\it (Freedman (1965) \cite{Freedman65})}\\
In the setting of example~\ref{ex:freedman63}, denote the space
of all distributions on $\Lambda$ by $\pi(\Lambda)$. Note that
since $\Lambda$ is Polish, so is $\pi(\Lambda)$ and so is the
product $\Lambda\times\pi(\Lambda)$.
\begin{theorem}
\label{thm:freedman}{\it (Freedman (1965) \cite{Freedman65})}\\
Let $X_1,X_2,\ldots$ form an sample of \iid-$P_0$ random integers,
let $\Lambda$
denote the space of all distributions on $\NN$ and let $\pi(\Lambda)$
denote the space of all Borel probability measures on $\Lambda$, both
in Prohorov's weak topology. The set of pairs
$(P_0,\Pi)\in\Lambda\times\pi(\Lambda)$
such that for all open $U\subset\Lambda$,
\[
  \limsup_{n\to\infty} P_0^n\Pi(U|\samplen) = 1,
\]
is residual.
\end{theorem}
And so, the set of pairs $(P_0,\Pi)\in\Lambda\times\pi(\Lambda)$
for which the limiting behaviour of the posterior is acceptable
to the frequentist, is \emph{meagre} in $\Lambda\times\pi(\Lambda)$. 
The proof relies on the following construction: for $k\geq1$,
define $\Lambda_k$ to be the subset of all probability distributions
$P$ on $\NN$ such that $P(X=k)=0$. Also define $\Lambda_0$ as the
union of all $\Lambda_k$, ($k\geq1$). Pick
$Q\in\Lambda\setminus\Lambda_0$. We assume that
$P_0\in\Lambda\setminus\Lambda_0$ and $P_0\neq Q$. Place a prior $\Pi_0$
on $\Lambda_0$ and choose
$\Pi=\ft12\Pi_0+\ft12\delta_Q$. Because $\Lambda_0$
is dense in $\Lambda$, priors of this type have full support in
$\Lambda$. But $P_0$ has full support in $\NN$ so for every
$k\in\NN$, $P_0^\infty(\exists_{m\geq1}:X_m=k)=1$: note that if
we observe $X_m=k$, the likelihood equals zero on $\Lambda_k$
so that $\Pi(\Lambda_k|\samplen)=0$ for all $n\geq m$,
$P_0^\infty$-almost-surely. Freedman shows this eliminates all of
$\Lambda_0$ asymptotically, if $\Pi_0$ is chosen in a suitable way,
forcing all posterior mass onto the point $\{Q\}$.
(See also, Le~Cam (1986) \cite{LeCam86}, section~17.7).\closebox
\end{example}

The question remains how Freedman's inconsistent posteriors relate
to the work presented here. Since test sequences of exponential power
exist to separate complements of weak neighbourhoods, \cf\
proposition~\ref{prop:weaktests}, Freedman's inconsistencies must violate the
requirement of remote contiguity in theorem~\ref{thm:consistency}. 
\begin{example}
\label{ex:contfreedman65}
As noted already, $\Lambda$ is a Polish space; in particular
$\Lambda$ is metric and second countable, so the subspace
$\Lambda\setminus\Lambda_0$ contains
a countable dense subset $D$. For $Q\in D$, let $V$ be the set of
all prior probability measures on $\Lambda$ with finite support, of which
one point is $Q$ and the remaining points lie in $\Lambda_0$. The proof
of the theorem in \cite{Freedman65} that asserts that the set of 
consistent pairs $(P_0,\Pi)$ is of the first category in
$\Lambda\times\pi(\Lambda)$ departs from the observation that if
$P_0$ lies in $\Lambda\setminus\Lambda_0$ and we use a prior from
$V$, then,
\[
  \Pi(\{Q\}|\samplen)\convas{P_0}1,
\]
(in fact, as is shown below, with $P_0^{\infty}$-probability one there exists
an $N\geq1$ such that $\Pi(\{Q\}|\samplen)=1$ for all $n\geq N$).
The proof continues to assert that $V$ lies dense in $\pi(\Lambda)$,
and, through sequences of continuous extensions involving $D$, that
posterior inconsistency for elements of $V$ implies posterior inconsistency
for all $\Pi$ in $\pi(\Lambda)$ with the possible exception of a set
of the first category.

From the present perspective it is interesting to view the
inconsistency of elements of $V$ in light of the conditions of
theorem~\ref{thm:consistency}. Define, for some bounded
$f:\NN\to\RR$ and $\ep>0$, two subsets of $\Lambda$,
\[
  B = \{ P: |Pf-P_0f|<\ft12\ep\},\quad
  V = \{ P: |Pf-P_0f|\geq\ep \}.
\]
Proposition~\ref{prop:weaktests} asserts the existence of a uniform test sequence for
$B$ versus $V$ of exponential power. With regard to remote contiguity,
for an element $\Pi$ of $V$ with support of order $M+1$, write,
\[
  \Pi = \beta \delta_Q + \sum_{m=1}^M \alpha_m \delta_{P_m},
\]
where $\beta+\sum_{m}\alpha_m=1$ and $P_m\in\Lambda_0$ ($1\leq m\leq M$). 
Without loss of generality, assume that $\ep$ and $f$ are such that
$Q$ does not lie in $B$. Consider,
\[
  \frac{dP_n^{\Pi|B}}{dP_0^n}(\samplen)
    = \frac{1}{\Pi(B)}\int_B \frac{dP^n}{dP_0^n}(\samplen)\,d\Pi(P)
    \leq \frac{1}{\Pi(B)}\sum_{m=1}^M \alpha_m 
      \frac{dP_m^n}{dP_0^n}(\samplen).
\]
For every $1\leq m\leq M$, there exists a $k(m)$ such that $P_m(X=k(m))=0$,
and the probability of the event $E_n$ that none of the
$X_1,\ldots,X_n$ equal $k(m)$ is $(1-P_0(X=k(m)))^n$. Note that
$E_n$ is also the event that ${dP_m^n}/{dP_0^n}(\samplen)>0$.

Hence for every $1\leq m\leq M$ and all $\sample$ in an event
of $P_0^\infty$-probability one, there exists an $N_m\geq1$ such that
${dP_m^n}/{dP_0^n}(\samplen)=0$ for all $n\geq N_m$. Consequently,
for all $\sample$ in an event of $P_0^\infty$-probability one,
there exists an $N\geq1$ such that
${dP_n^{\Pi|B}}/{dP_0^n}(\samplen)=0$ for all $n\geq N$. Therefore,
condition {\it (ii)} of lemma~\ref{lem:rcfirstlemma} is not satisfied
for any sequence $a_n\downarrow0$. A direct proof that (\ref{eq:defrc})
does not hold for any $a_n$ is also possible:
given the prior $\Pi\in V$, define,
\[
  \phi_n(\samplen) = \prod_{m=1}^M 1_{\{ \exists_{1\leq i\leq n} : X_i=k(m) \}}.
\]
Then the expectation of $\phi_n$ with respect to the local prior
predictive distribution equals zero, so $P_n^{\Pi|B}\phi_n=o(a_n)$ for
any $a_n\downarrow0$. However, $P_0^n\phi_n(\samplen)\to1$, so the
prior $\Pi$ does {\it not} give rise to a sequence of prior predictive
distributions $(P_n^{\Pi|B})$ with respect to which $(P_0^n)$ is remotely
contiguous, for any $a_n\downarrow0$.\closebox
\end{example}

\subsection{Consistency, Bayesian tests and the Hellinger metric}
\label{sub:estimation}

Let us first consider characterization of posterior consistency
in terms of the family of real-valued functions on the parameter
space that are bounded and continuous.
\begin{proposition}
\label{prop:prokhorov}
Assume that $\Tht$ is a Hausdorff, completely regular space. The
posterior is consistent at $\tht_0\in\Tht$, if and only if, 
\begin{equation}
  \label{eq:prokhorov}
  \int f(\tht)\,d\Pi(\tht|\samplen) \convprob{P_{\tht_0,n}} f(\tht_0),
\end{equation}
for every bounded, continuous $f:\Tht\to\RR$.
\end{proposition}
\begin{proof} 
Assume (\ref{eq:cons}). Let $f:\Tht\rightarrow\RR$ be
bounded and continuous (with $M>0$ such that $|f|\leq M$).
Let $\eta>0$ be given and let $U\subset\Tht$ be a neighbourhood
of $\tht_0$ such that $|f(\tht)-f(\tht_0)|<\eta$ for all $\tht\in U$.
Integrate $f$ with respect to the ($P_{\tht_0,n}$-almost-surely well-defined)
posterior and to $\delta_{\tht_0}$:
\[
  \begin{split}
  \Bigl|\int &f(\tht)\,d\Pi(\tht|\samplen) -f(\tht_0)\Bigr|\\
  &\leq \int_{\Tht\setminus U} |f(\tht)-f(\tht_0)
    |\,d\Pi(\tht|\samplen)
    + \int_{U} |f(\tht)-f(\tht_0)|\,d\Pi(\tht|\samplen)\\
  &\leq 2M\,\Pi(\Tht\setminus U|\samplen)
  \quad+ \sup_{\tht\in U}|f(\tht)-f(\tht_0)|\,\Pi(U|\samplen)
  \leq \eta + o_{P_{\tht_0,n}}(1),
  \end{split}
\]
as $n\rightarrow\infty$, so that (\ref{eq:prokhorov}) holds.
Conversely, assume (\ref{eq:prokhorov}). Let
$U$ be an open neighbourhood of $\tht_0$. Because $\Tht$ is completely
regular, there exists a continuous $f:\Tht\rightarrow[0,1]$
such that $f=1$ at $\{\tht_0\}$ and $f=0$ on $\Tht\setminus U$. Then,
\[
   \Pi(U|\samplen) 
    \geq \int f(\tht)\,d\Pi(\tht|\samplen) \convprob{P_{\tht_0,n}}
    \int f(\tht)\,d\delta_{\tht_0}(P) = 1.
\]
Consequently, (\ref{eq:cons}) holds.
\end{proof}
Proposition~\ref{prop:prokhorov} is used to prove consistency of
frequentist point-estimators derived from the posterior.
\begin{example}
\label{ex:ptestimator}
Consider a model $\scrP$ of single-observation distributions
$P$ on $(\scrX,\scrB)$ for \iid\ data $(X_1,X_2,\ldots,X_n)\sim P^n$,
$(n\geq1)$. Assume that the
true distribution of the data is $P_0\in\scrP$ and that the model
topology is Prohorov's weak topology or stronger. Then for any
bounded, continuous $g:\scrX\to\RR$, the map,
\[
  f:\scrP\to\RR:P\mapsto \bigl| (P - P_0) g(X) \bigr|,
\]
is continuous. Assuming that the posterior is weakly consistent at $P_0$,
\begin{equation}
  \label{eq:pointestimators}
  \bigl| P_1^{\Pi_n|\samplen}g - P_0g\bigr|
    \leq \int \bigl|(P-P_0) g\bigr| \,d\Pi(P|\samplen)
    \convprob{P_{\tht_0}} 0,
\end{equation}
so posterior predictive distributions are consistent
point estimators in Prohorov's weak topology. Replacing the maps
$g$ by bounded, measurable maps
$\scrX\to\RR$ and assuming posterior consistency in $\scrT_1$,
one proves consistency of posterior predictive distributions
in $\scrT_1$ in exactly the same
way. Taking the supremum over measurable $g:\scrX\to[0,1]$ in
(\ref{eq:pointestimators}) and assuming that the posterior is
consistent in the total variational topology,
posterior predictive distributions are consistent in total
variation as frequentist point estimators. \closebox
\end{example}

The vast majority of non-parametric applications of Bayesian
methods in the literature is based on the intimate relation that
exists between testing and the Hellinger metric (see
\cite{LeCam86}, section~16.4). Proofs concerning posterior
consistency or posterior convergence at a rate rely on the
existence of tests for small parameter subsets $B_n$ surrounding
a point $\tht_0\in\Tht$, versus the complements $V_n$ of
neighbourhoods of the point $\tht_0$. The building block in such
constructions is the following application of the minimax theorem.
\begin{proposition}
\label{prop:minmaxhell}{\it (Minimax Hellinger tests)}\\
Consider a model $\scrP$ of single-observation distributions $P$
for \iid\ data. 
Let $B,V\subset\scrP$ be convex with $H(B,V)>0$.
There exists a test sequence $(\phi_n)$ such that,
\[
  \sup_{P\in B} P^n\phi_n\leq e^{-nH^2(B,V)},\,
  \sup_{Q\in V} Q^n(1-\phi_n) \leq e^{-nH^2(B,V)}.
\]
\end{proposition}
\begin{proof}
This is an application of the minimax theorem. See Le~Cam
(1986) \cite{LeCam86}, section~16.4 for details. 
\end{proof}
Questions concerning consistency require the existence of tests
in which at least one of the two hypotheses is a non-convex set,
typically the complement of a neighbourhood. Imposing the model
$\scrP$ to be of bounded entropy with respect to the Hellinger
metric allows construction of such tests, based on the
uniform tests of proposition~\ref{prop:minmaxhell}.
Below, we apply well-known constructions for the uniform tests
in Schwartz's theorem from the frequentist literature
\cite{LeCam73,Birge83,Birge84,Ghosal00}
to the construction of Bayesian tests.
Due to relations that exist between metrics for model parameters
and the Hellinger metric in many examples and applications, the
material covered here is widely applicable in (non-parametric)
models for \iid\ data.
\begin{example}
\label{ex:entropy}
Consider a model $\scrP$ of distributions $P$
for \iid\ data $X^n\sim P^n$, $(n\geq1)$ and,
in addition, suppose that $\scrP$ is totally bounded with respect
to the Hellinger distance. Let $P_0\in\scrP$ and $\ep>0$ be given,
denote $V(\ep)=\{P\in\scrP:H(P_0,P)\geq4\ep\}$,
$B_H(\ep)=\{P\in\scrP:H(P_0,P)<\ep\}$. There exists an $N(\ep)\geq1$ and a cover
of $V(\ep)$ by $H$-balls $V_1,\ldots,V_{N(\ep)}$ of radius $\ep$ and
for any point $Q$ in any $V_i$ and any $P\in B_H(\ep)$,
$H(Q,P)>2\ep$. According to proposition~\ref{prop:barycentres} with
$\al=1/2$ and (\ref{eq:bayeshelltests}), for
each $1\leq i\leq N(\ep)$ there exists a Bayesian test sequence
$(\phi_{i,n})$ for $B_H(\ep)$ versus $V_i$ of power
(upper bounded by) $\exp(-2n\ep^2)$. Then, for any subset
$B'\subset B_H(\ep)$,
\begin{equation}
  \label{eq:sepbound}
  \begin{split}
  P_{n}^{\Pi|B'}\Pi(&V|\samplen)\leq
  \sum_{i=1}^{N(\ep)}P_{n}^{\Pi|B'}\Pi(V_i|\samplen)\\
  &\leq
  \frac1{\Pi(B')}\sum_{i=1}^{N(\ep)}\Bigl(\int_{B'}P^n\phi_n\,d\Pi(P) 
    + \int_{V_i}P^n(1-\phi_n)\,d\Pi(P)\Bigr)\\
  &\leq \sum_{i=1}^{N(\ep)} \sqrt{\frac{\Pi(V_i)}{\Pi(B')}}\exp(-2n\ep^2),
  \end{split}
\end{equation}
which is smaller than or equal to $e^{-n\ep^2}$ for large enough $n$.
If $\ep=\ep_n$ with $\ep_n\downarrow0$
and $n\ep_n^2\to\infty$, and the model's
Hellinger entropy is upper-bounded by $\log N(\ep_n,\scrP,H)\leq Kn\ep_n^2$
for some $K>0$, the construction extends to tests that separate
$V_n=\{P\in\scrP:H(P_0,P)\geq4\ep_n\}$ from
$B_n=\{P\in\scrP:H(P_0,P)<\ep_n\}$ asymptotically, with power
$\exp(-nL\ep_n^2)$ for some $L>0$. (See also
the so-called \emph{Le~Cam dimension} of a model \cite{LeCam73}
and Birg\'e's rate-oriented work \cite{Birge83,Birge84}.)
It is worth pointing out at this stage that posterior inconsistency
due to the phenomenon of `data tracking' \cite{Barron99,Walker05},
whereby weak posterior consistency holds but
Hellinger consistency fails, can only be due to failure of the
testing condition in the Hellinger case.

Note that the 
argument also extends to models that are Hellinger separable: in that
case (\ref{eq:sepbound}) remains valid, but with $N(\ep)=\infty$. The
mass fractions $\Pi(V_i)/\Pi(B')$ become
important (we point to strong connections with Walker's theorem
\cite{Walker04,Walker07}). Here we see the balance between
prior mass and testing power for Bayesian tests, as intended by
the remark that closes the subsection on the existence of Bayesian
test sequences in section~\ref{sec:post}. 
\closebox
\end{example}

To balance entropy and prior mass differently in Hellinger separable models,
Barron (1988) \cite{Barron88} and Barron \ea\ (1999) \cite{Barron99}
formulate an alternative condition that is based on the Radon property
that any prior on a Polish space has.
\begin{example}
\label{ex:barron}
Consider a model $\scrP$ of distributions $P$
for \iid\ data $X^n\sim P^n$, $(n\geq1)$, with priors
$(\Pi_n)$. Assume that the model $\scrP$ is Polish in the Hellinger
topology. Let $P_0\in\scrP$ and $\ep>0$ be given; for a fixed $M>1$, define
$V=\{P\in\scrP:H(P_0,P)\geq M\ep\}$,
$B_H=\{P\in\scrP:H(P_0,P)<\ep\}$.
For any sequence $\delta_m\downarrow0$, there exist compacta
$K_m\subset\scrP$ such that $\Pi(K_m)\geq 1-\delta_m$ for all $m\geq1$.
For each $m\geq1$, $K_m$ is Hellinger totally bounded so there exists
a Bayesian test sequence $\phi_{m,n}$ for $B_H(\ep)\cap K_m$ versus
$V(\ep)\cap K_m$. Since,
\[
  \begin{split}
    \int_{B_H} P^n&\phi_n\,d\Pi(P)
    + \int_{V} Q^n(1-\phi_n)\,d\Pi(Q)\\
      &\leq \int_{B_H\cap K_m} P^n\phi_{m,n}\,d\Pi(P)
      \quad + \int_{V\cap K_m} Q^n(1-\phi_{m,n})\,d\Pi(Q) + \delta_m,
  \end{split}
\]
and all three terms go to zero, a diagonalization argument confirms the
existence of a Bayesian test for $B_H$ versus $V$.
To control the power of this test and to
generalise to the case where $\ep=\ep_n$ is $n$-dependent,
more is required: as we increase $m$ with $n$, the prior mass $\delta_{m(n)}$
outside of $K_n=K_{m(n)}$ must decrease fast enough, while the order of the
cover must be bounded: if $\Pi_n(K_n)\geq1-\exp(-L_1n\ep_n^2)$ and the
Hellinger entropy of $K_n$ satisfies
$\log N(\ep_n,K_n,H)\leq L_2n\ep_n^2$ for some $L_1,L_2>0$, there
exist $M>1$, $L>0$, and a sequence of tests $(\phi_n)$ such that,
\[
  \int_{B_H(\ep_n)} P^n\phi_n\,d\Pi(P) + \int_{V(\ep_n)} Q^n(1-\phi_n)\,d\Pi(Q)
      \leq e^{-Ln \ep_n^2},
\]
for large enough $n$. (For related constructions, see Barron (1988)
\cite{Barron88}, Barron \ea\ (1999) \cite{Barron99} and Ghosal, Ghosh
and van~der~Vaart (2000) \cite{Ghosal00}.)
\end{example}
To apply corollary~\ref{cor:schwartz} consider the following steps.
\begin{example}
\label{ex:schwartz}
As an example of the tests required under condition~{\it (i)} of
corollary~\ref{cor:schwartz}, consider $\scrP$ in the Hellinger
topology, assuming totally-boundedness.
Let $U$ be the Hellinger-ball of radius $4\ep$ around $P_{\tht_0}$ of
example~\ref{ex:entropy} and let $V$ be its complement. The Hellinger
ball $B_H(\ep)$ in equation~(\ref{eq:sepbound}) contains the set
$K(\ep)$. Alternatively
we may consider the model in any of the weak topologies $\scrT_n$:
let $\ep>0$ be given and let $U$ denote a weak neighbourhood of the form
$\{ P\in\scrP\,:\,|(P^n-P^n_0)f|\geq2\ep\}$, for some bounded measurable
$f:\scrX_n\to[0,1]$, as in proposition~\ref{prop:weaktests}. The set $B$
of proposition~\ref{prop:weaktests} contains a set $K(\delta)$, for some
$\delta>0$. Both these applications were noted by Schwartz in
\cite{Schwartz65}.
\end{example}

\subsection{Some examples of remotely contiguous sequences}

The following two examples illustrate the difference
between contiguity and remote contiguity in the context of
parametric and non-parametric regression.
\begin{example}
\label{ex:regression}
Let $\scrF$ denote a class of functions $\RR\to\RR$. We consider
samples
$\samplen=((X_1,Y_1),\ldots,(X_n,Y_n))$,
($n\geq1$) of points in $\RR^2$, assumed to be related through
$Y_i = f_0(X_i) + e_i$ for some unknown $f_0\in\scrF$, where the
errors are \iid\ standard normal $e_1,\ldots,e_n\sim N(0,1)^n$
and independent of the \iid\ covariates $X_1,\ldots,X_n\sim P^n$,
for some (ancillary) distribution $P$ on $\RR$. It is assumed
that $\scrF\subset L^2(P)$ and we use the $L^2$-norm
$\|f\|^2_{P,2}=\int f^2\,dP$ to define a metric $d$ on $\scrF$,
$d(f,g)=\|f-g\|_{P,2}$. Given a parameter $f\in\scrF$,
denote the sample distributions as $P_{f,n}$. We distinguish
two cases: {\it (a)} the case of linear regression,
where $\scrF=\{f_\tht:\RR\to\RR:\tht\in\Tht\}$, where
$\tht=(a,b)\in\Tht=\RR^2$ and $f_\tht(x)=ax+b$; and {\it (b)}
the case of non-parametric regression, where we do not restrict
$\scrF$ beforehand. 

Let $\Pi$ be a Borel prior $\Pi$ on $\scrF$ and place remote
contiguity in context by assuming, for the moment,
that for some $\rho>0$, there exist $0<r<\rho$
and $\tau>0$, as well as Bayesian tests $\phi_n$ for
$B=\{f\in\scrF:\|f-f_0\|_{P,2}<r\}$ versus
$V=\{f\in\scrF:\|f-f_0\|_{P,2}\geq\rho\}$ under $\Pi$
of power $a_n=\exp(-\ft12n\tau^2)$. If this is the case, we may
assume that $r<\ft12\tau$ without loss of generality.
Suppose also that $\Pi$ has a support
in $L^2(P)$ that contains all of $\scrF$.

Let us concentrate on case {\it (b)} first: a bit of manipulation
casts the $a_n$-rescaled likelihood ratio for $f\in\scrF$ in the
following form,
\begin{equation}
  \label{eq:likelihood}
  a_n^{-1}\frac{dP_{f,n}}{dP_{f_0,n}}(\samplen)
  = e^{ -\ft12\sum_{i=1}^n\left( e_i(f-f_0)(X_i)+(f-f_0)^2(X_i) -\tau^2\right)},
\end{equation}
under $\samplen\sim P_{f_0,n}$. The exponent is controlled by the
law of large numbers,
\[
  \frac{1}{n}\sum_{i=1}^n\bigl( e_i(f-f_0)(X_i)+(f-f_0)^2(X_i)-\tau^2\bigr)
  \convas{P_{f_0}^\infty} \|f-f_0\|^2_{P,2} - \tau^2.
\]
Hence, for every $\ep>0$ there exists an $N(f,\ep)\geq1$ such that
the exponent in (\ref{eq:likelihood}) satisfies the upper bound,
\[
  \sum_{i=1}^n\bigl( e_i(f-f_0)(X_i)+(f-f_0)^2(X_i)-\tau^2 \bigr)
  \leq n(\|f-f_0\|^2_{P,2} - \tau^2 + \ep^2),
\]
for all $n\geq N(f,\ep)$. Since $\Pi(B)>0$, we may condition $\Pi$ 
on $B$, choose $\ep=\ft12\tau$ and use Fatou's inequality to find that,
\[
  \liminf_{n\to\infty} e^{\ft12n\tau^2}
    \frac{dP_n^{\Pi|B}}{dP_{f_0,n}}(\samplen)
    \geq \liminf_{n\to\infty} e^{\ft14n\tau^2}=\infty,
\]
$P^\infty_{f_0}$-almost-surely. Consequently,
for any choice of $\delta$,
\[
  P_{f_0,n}\biggl( \frac{dP_n^{\Pi|B}}{dP_{f_0,n}}(\samplen)
    < \delta\,e^{-\ft12n\tau^2}\biggr) \to 0,
\]
and we conclude that $P_{f_0,n}\ctg e^{\ft12n\tau^2}P_n^{\Pi|B}$.
Based on theorem~\ref{thm:consistency}, we conclude that,
\[
  \Pi\bigl(\,\|f-f_0\|_{P,2}<\rho \bigm|\samplen\bigr)
    \convprob{P_{f_0,n}}1,
\]
\ie\ posterior consistency for the regression function
in $L^2(P)$-norm obtains.
\closebox
\end{example}

\begin{example}
\label{ex:rcLAN}
As for case {\it (a)}, one has the choice
of using a prior like above, but also to proceed differently: 
expression (\ref{eq:likelihood}) can be written in terms of a local
parameter $h\in\RR^k$ which, for given $\tht_0$ and $n\geq1$, is
related to
$\tht$ by $\tht=\tht_0+n^{-1/2}h$. For $h\in\RR^2$, we write
$P_{h,n}=P_{\tht_0+n^{-1/2}h,n}$, $P_{0,n}=P_{\tht_0,n}$
and rewrite the likelihood ratio (\ref{eq:likelihood}) as follows,
\begin{equation}
  \label{eq:LAN}
  \frac{dP_{h,n}}{dP_{0,n}}(\samplen)
    = e^{\frac{1}{\sqrt{n}}\sum_{i=1}^n h\cdot\ell_{\tht_0}(X_i,Y_i)
      -\frac12 h\cdot I_{\tht_0}\cdot h + R_n},
\end{equation}
where $\ell_{\tht_0}:\RR^2\to\RR^2:(x,y)\mapsto(y-a_0x-b_0)(x,1)$
is the score function for $\tht$,
$I_{\tht_0}=P_{\tht_0,1}\ell_{\tht_0}\ell_{\tht_0}^T$ is the Fisher
information matrix and $R_n\convprob{P_{\tht_0,n}}0$. Assume that
$I_{\tht_0}$ is non-singular and note the
central limit,
\[
  \frac{1}{\sqrt{n}}\sum_{i=1}^n \ell_{\tht_0}(X_i,Y_i)
    \convweak{P_{\tht_0,n}} N_2(0,I_{\tht_0}),
\]
which expresses local asymptotic normality of
the model \cite{LeCam60b} and implies that for any fixed $h\in\RR^2$,
$P_{h,n}\ctg P_{0,n}$.
\begin{lemma}
\label{lem:rcLAN}
Assume that the model satisfies LAN condition
(\ref{eq:LAN}) with non-singular $I_{\tht_0}$ and
that the prior $\Pi$ for $\tht$ has a Lebesgue-density
$\pi:\RR^d\to\RR$ that is continuous and strictly positive in
all of $\Tht$. For given $H>0$, define the subsets
$B_n=\{ \tht\in\Tht:\tht=\tht_0+n^{-1/2}h,\|h\|\leq H\}$.
Then,
\begin{equation}
  \label{eq:LANrcctg}
    P_{0,n}\ctg c_n^{-1}P_n^{\Pi|B_n},
\end{equation}
for any $c_n\downarrow0$.
\end{lemma}
\begin{proof}
According to lemma~3 in section~8.4 of Le~Cam and Yang (1990)
\cite{LeCam90}, $P_{\tht_0,n}$ is contiguous with respect
to $P_n^{\Pi|B_n}$. That implies the assertion.
\end{proof}
Note that for some $K>0$, $\Pi(B_n)\geq b_n:=K(H/\sqrt{n})^d$.
Assume again the existence of Bayesian tests for
$V=\{\tht\in\Tht:\|\tht-\tht_0\|>\rho\}$ (for some $\rho>0$)
versus $B_n$ (or some $B$ such that $B_n\subset B$), of
power $a_n=\exp(-\ft12n\tau^2)$ (for some $\tau>0$).
Then $a_n b_n^{-1}=o(1)$, and, assuming (\ref{eq:LANrcctg}),
theorem~\ref{thm:rates} implies that
$\Pi(\|\tht-\tht_0\|>\rho|\samplen)\convprob{P_{\tht_0,n}}0$,
so consistency is straightforwardly demonstrated.

The case becomes somewhat more complicated if we are interested in
optimality of parametric rates: following the above, a logarithmic
correction arises from the lower bound
$\Pi(B_n)\geq K(H/\sqrt{n})^d$ when combined in the
application of theorem~\ref{thm:rates}. To alleviate this, we
adapt the construction somewhat: define
$V_n=\{\tht\in\Tht: \|\tht-\tht_0\|\leq M_n\,n^{-1/2}\}$ for some
$M_n\to\infty$ and $B_n$ like above. Under the condition that there
exists a uniform test sequence for any \emph{fixed}
$V=\{\tht\in\Tht:\|\tht-\tht_0\|>\rho\}$ versus $B_n$ (see,
for example, \cite{Kleijn12}), uniform test sequences for 
$V_n$ versus $B_n$ of power $e^{-K'M_n^2}$ exist, for some $k'>0$.
Alternatively, assume that the Hellinger distance and the norm on $\Tht$
are related through inequalities of the form,
\[
  K_1\|\tht-\tht'\|\leq H(P_{\tht},P_{\tht'})\leq K_2\|\tht-\tht'\|,
\]
for some constants $K_1,K_2>0$. Then cover $V_n$ with rings,
\[
  V_{n,k}=\biggl\{ \tht\in V_n: \frac{(M_n+k-1)}{\sqrt{n}}
    \leq\|\tht-\tht_0\| \leq \frac{(M_n+k)}{\sqrt{n}}\biggr\},
\]
for $k\geq1$ and cover each ring with balls $V_{n,k,l}$
of radius $n^{-1/2}$, where $1\leq l\leq L_{n,k}$ and $L_{n,k}$
the minimal number of radius-$n^{-1/2}$ balls needed to cover $V_{n,k}$,
related to the \emph{Le~Cam dimension} \cite{LeCam73}. With the $B_n$
defined like above, and the inequality,
\[
  \begin{split}
  \int P_{\tht,n}\Pi(&V_{n,k,l}|\samplen)\,d\Pi_n(\tht|B_n)\\
  &\leq \sup_{\tht\in B_n} P_{\tht,n}\phi_{n,k,l} +
    \frac{\Pi_n(V_{n,k,l})}{\Pi_n(B_n)}
    \sup_{\tht\in V_{n,k,l}} P_{\tht,n}(1-\phi_{n,k,l}),
  \end{split}
\]
where the $\phi_{n,k,l}$ are the uniform minimax tests for $B_n$
versus $V_{n,k,l}$ of lemma \ref{prop:minmaxhell}, of power
$\exp(-K'(M_n+k-1)^2)$ for some $K'>0$. We define
$\phi_{n,k}=\max\{\phi_{n,k,l}:1\leq l\leq L_{n,k}\}$ for $V_{n,k}$
versus $B_n$ and note,
\[
  \int P_{\tht,n}\Pi(V_{n,k}|\samplen)\,d\Pi_n(\tht|B_n)
  \leq \Bigl(L_{n,k} +
    \frac{\Pi_n(V_{n,k})}{\Pi_n(B_n)}\Bigr) e^{-K(M_n+k-1)^2},
\]
where the numbers $L_{n,k}$ are upper bounded by a multiple of
$(M_n+k)^{d}$ and the fraction of
prior masses $\Pi_n(V_{n,k})/\Pi_n(B_n)$ can be controlled without
logarithmic corrections when summing over $k$ next. \closebox

But remote contiguity also applies in more irregular situations: 
example~\ref{ex:noKLpriors} does not
admit KL priors, but satisfies the requirement of remote
contiguity. (Choose $\eta$ equal to the uniform density for
simplicity.)
\begin{example}
\label{ex:rcnoKLpriors}
Consider $X_1,X_2,\ldots$ that form an \iid\ sample from the uniform
distribution on $[\tht,\tht+1]$, for unknown $\tht\in\RR$. The model
is parametrized in terms of distributions $P_\tht$ with Lebesgue
densities of the form $p_{\tht}(x)=1_{ [\tht,\tht+1] }(x)$,
for $\tht\in\Tht=\RR$. Pick a prior $\Pi$ on $\Tht$ with a continuous
and strictly positive Lebesgue density $\pi:\RR\to\RR$ and, for some
rate $\delta_n\downarrow0$, choose $B_n=(\tht_0,\tht_0+\delta_n)$. Note that
for any $\al>0$, there exists an $N\geq1$ such that for all $n\geq N$,
$(1-\al)\pi(\tht_0)\delta_n\leq\Pi(B_n)\leq(1+\al)\pi(\tht_0)\delta_n$.
Note that for any $\tht\in B_n$ and $\samplen\sim
P_{\tht_0}^n$, $dP_{\tht}^n/dP_{\tht_0}^n(\samplen)=1\{ X_{(1)}>\tht \}$,
and correspondingly,
\[
  \begin{split}
  \frac{dP_n^{\Pi|B_n}}{dP_{\tht_0}^n}(\samplen)
    &= \Pi_n(B_n)^{-1}
    \int_{\tht_0}^{\tht_0+\delta_n}1\{ X_{(1)}>\tht \}\,d\Pi(\tht)\\
    &\geq \frac{1-\al}{1+\al}\frac{\delta_n\wedge(X_{(1)}-\tht_0)}{\delta_n},
  \end{split}
\]
for large enough $n$. As a consequence, for every $\delta>0$ and
all $a_n\downarrow0$,
\[
  P_{\tht_0}^n\biggl(\frac{dP_n^{\Pi|B_n}}{dP_{\tht_0}^n}(\samplen)
    < \delta\,a_n \biggr)
  \leq
    P_{\tht_0}^n\bigl(
    \,\delta_n^{-1}(X_{(1)}-\tht_0) < (1+\al) \delta\,a_n \,\bigr),
\]
for large enough $n\geq1$. Since $n(X_{(1)}-\tht_0)$ has an exponential
weak limit under $P_{\tht_0}^n$, we choose $\delta_n=n^{-1}$, so that
the \rhs\ in the above display goes to zero.
So $P_{\tht_0,n} \ctg a_n^{-1}P_n^{\Pi_n|B_n}$, for any $a_n\downarrow0$.
\closebox
\end{example}
To show consistency and derive the posterior rate of convergence in
example~\ref{ex:noKLpriors}, we use theorem~\ref{thm:rates}.
\begin{example}
\label{ex:testnoKLpriors}
Continuing with example~\ref{ex:rcnoKLpriors}, we define
$V_n=\{\tht: \tht-\tht_0>\ep_n \}$. It is noted that, for every
$0<c<1$, the likelihood ratio test,
\[
  \phi_n(X^n)
  = 1\{dP_{\tht_0+\ep_n,n}/dP_{\tht_0,n}(\samplen)>c\} 
    = 1\{ X_{(1)}>\tht_0+\ep_n \},
\]
satisfies $P_{\tht}^n(1-\phi_n)(\samplen)=0$ for all $\tht\in V_n$,
and if we choose $\delta_n=1/2$ and $\ep_n=M_n/n$ for some $M_n\to\infty$,
$P_{\tht}^n\phi_n\leq e^{-M_n+1}$ for all $\tht\in B_n$, so that,
\[
  \int_{B_n}P_{\tht}^n\phi_n(\,d\Pi(\tht)
    + \int_{V_n}P_{\tht}^n(1-\phi_n)\,d\Pi(\tht)
  \leq \Pi(B_n)\,e^{-M_n+1},
\]
Using lemma~\ref{lem:testineq}, we see that
$P_n^{\Pi|B_n}\Pi(V_n|\samplen)\leq e^{-M_n+1}$. Based on the
conclusion of example~\ref{ex:rcnoKLpriors} above, remote contiguity
implies that $P_{\tht_0}^n\Pi(V_n|\samplen)\to0$. Treating the
case $\tht<\tht_0-\ep_n$ similarly, we conclude that the posterior
is consistent at (any $\ep_n$ slower than) rate $1/n$. 
\end{example}

To conclude, we demonstrate the relevance of priors satisfying the
lower bound (\ref{eq:GGV}). Let us repeat lemma~8.1 in \cite{Ghosal00},
to demonstrate that the sequence $(P_0^n)$ is remotely contiguous
with respect to the local prior predictive distributions based on the $B_n$
of example~\ref{ex:GGV}.
\begin{lemma}
\label{lem:ctgGGV}
For all $n\geq1$, assume that $(X_1, X_2, \ldots, X_n)\in\scrX^n\sim P_0^n$
for some $P_0\in\scrP$ and let $\ep_n\downarrow0$ be given. Let $B_n$
be as in example~\ref{ex:GGV}. Then, for any priors $\Pi_n$ such that
$\Pi_n(B_n)>0$,
\[
    P_{\tht_0,n}\biggl( \,\int
    \frac{dP_{\tht}^n}{dP_{\tht_0}^n}(\samplen)\,d\Pi_n(\tht|B_n)
    < e^{-cn\ep_n^2}\biggr)\to0,
\]
for any constant $c>1$.
\end{lemma}
\end{example}

\subsection{The sparse normal means problem}

For an example of consistency in the false-detection-rate (FDR) sense,
we turn to the most
prototypical instance of sparsity, the so-called \emph{sparse normal
means problem}: in recent years various types of priors have been
proposed for the Bayesian recovery of a nearly-black vector in the
Gaussian sequence model. Most intuitive in this context is the
class of \emph{spike-and-slab priors} \cite{Mitchell88}, which
first select a sparse subset of non-zero components and then
draws those from a product distribution. But other
proposals have also been made,
\eg\ the horseshoe prior \cite{Carvalho10}, a scale-mixture of
normals. Below, we consider FDR-type consistency with
spike-and-slab priors.
\begin{example}
\label{ex:sparsenormalmeans}
Estimation of a nearly-black vector of locations
in the Gaussian sequence model is based on $n$-point samples
$\samplen=(X_1,\ldots,X_n)$ assumed distributed according to,
\begin{equation}
  \label{eq:normalmeans}
  X_i = \tht_i + \vep_i,
\end{equation}
(for all $1\leq i\leq n$), where $\vep_1,\ldots,\vep_n$
form an \iid\ sample of standard-normally distributed errors.
The parameter $\tht$ is a sequence $(\tht_i:i\geq1)$ in $\RR$, with
$n$-dimensional projection $\tht^n=(\tht_1,\ldots,\tht_n)$,
for every $n\geq1$. The corresponding distributions for
$\samplen$ are denoted $P_{\tht,n}$ for all $n\geq1$. 

Denoting by $p_n$ the number of non-zero components of the vector
$\tht^n=(\tht_1,\ldots,\tht_n)\in\RR^n$, sparsity is imposed through
the assumption that $\tht$ is \emph{nearly black}, that is,
$p_n\to\infty$, but $p_n=o(n)$ as $n\to\infty$. For any integer
$0\leq p\leq n$, denote the space
of $n$-dimensional vectors $\tht^n$ with exactly $p$ non-zero
components by $\ell_{0,n}(p)$. For later reference, we
introduce, for every subset $S$ of $I_n:=\{1,\ldots,n\}$, the space
$R_n^S:=\{\tht^n\in\RR^n:\tht_i=1\{i\in S\}\tht_i, 1\leq i\leq n\}$.

Popular sub-problems concern selection of the non-zero components
\cite{Birge01} and (subsequent) minimax-optimal estimation of the
non-zero components \cite{Donoho94} (especially with the LASSO in
related regression problems, see, for example, \cite{Zhang08}).
Many authors have followed Bayesian approaches; for empirical
priors, see \cite{Johnstone04}, and for hierarchical priors, see
\cite{Castillo12} (and references therein).

As $n$ grows, the minimax-rate at which the $L_2$-error
for estimation of $\tht^n$ grows, is bounded in the following,
sparsity-induced way \cite{Donoho92},
\[
  \inf_{\hat{\tht}^n} \sup_{\tht^n\in\ell_{0,n}(p_n)}
    P_{\tht,n} \bigl\| \hat{\tht}^n - \tht^n \bigr\|_{2,n}^2
  \leq 2p_n\log\frac{n}{p_n}(1+o(1)),
\]
as $n\to\infty$, where $\hat{\tht}^n$ runs over all estimators for
$\tht^n$.

A natural proposal for a prior $\Pi$ for $\tht$ \cite{Mitchell88}
(or rather, priors $\Pi_n$ for all $\tht^n$ ($n\geq1$)), is to draw
a sparse $\tht^n$ hierarchically \cite{Castillo12}:
given $n\geq1$, first draw $p\sim \pi_{n}$ (for some
distribution $\pi_{n}$ on $\{0,1,\ldots, n\}$), then draw
a subset $S$ of order $p$ from $\{1,\ldots, n\}$
uniformly at random, and draw $\tht^n$ by setting
$\tht_i=0$ if $i\not\in S$ and $(\tht_i:i\in S)\sim G^{p}$,
for some distribution $G$ on (all of) $\RR$. The components of
$\tht^n$ can therefore be thought of as having been drawn from
a mixture of a distribution degenerate at zero
(the \emph{spike}) and a full-support distribution $G$
(the \emph{slab}).

To show that methods presented in this paper also apply in
complicated problems like this, we give a proof of posterior
convergence in the FDR sense. We appeal freely to useful
results that appeared elsewhere, in particular in \cite{Castillo12}:
we adopt some of Castillo and van~der~Vaart's more technical
steps to reconstitute the FDR-consistency proof based on
Bayesian testing and remote contiguity: to compare, the testing
condition and prior-mass lower bound of theorem~\ref{thm:rates}
are dealt with simultaneously, while the remote contiguity
statement is treated separately. (We stress that only the way
of organising the proof, not the result is new. In fact, we
prove only part of what \cite{Castillo12} achieves.)

Assume that the data follows (\ref{eq:normalmeans})
and denote by $\tht_0$ the true vector of normal means.
For each $n\geq1$, let $p_n$ (respectively $p$) denote number of non-zero
components of $\tht_0^n$ (respectively $\tht^n$).
We do not assume that the true degree of sparsity $p_n$
is fully known, but for simplicity and brevity we assume that
there is a known sequence of upper bounds $q_n$, such that for some
constant $A>1$, $p_n\leq q_n\leq A\,p_n$, for all $n\geq1$.
(Indeed, theorem~2.1 in \cite{Castillo12} very cleverly shows that
if $G$ has a second moment and the prior density for the sparsity
level has a tail that is slim enough, 
then the posterior concentrates on sets of the form,
$\{\tht^n\in\RR^n\,:\,p \leq A p_n\}$ under $P_0$, for some $A>1$.)

Set $r_n^2=p_n\log(n/p_n)$ and define two subsets of $\RR^n$,
\[
  \begin{split}
  V_n &= \bigl\{ \tht^n:\,p \leq A p_n,\,
    \|\tht^n-\tht_0^n\|_{2,n} > M r_n
    \bigr\},\\
  B_n &= \bigl\{ \tht^n:\,
    \|\tht^n-\tht_0^n\|_{2,n} \leq d\, r_n,
    \bigr\},
  \end{split}
\]
assuming for future reference that $\Pi(B_n)>0$.
As for $V_n$, we split further: define, for all $j\geq1$,
\[
  V_{n,j} = \bigl\{ \tht^n\in V_n\,:\,
     jM r_n < \|\tht^n-\tht_0^n\|_{2,n}
       \leq (j+1)M r_n
    \bigr\}.
\]
Next, we subdivide $V_{n,j}$ into intersections with the spaces
$R_n^S$ for $S\subset I_n$: we write
$V_{n,j}=\cup\{V_{n,S,j}:S\subset I_n\}$ with
$V_{n,S,j} = V_{n,j}\cap R_n^S$. For every $n\geq1,j\geq1$
and $S\subset I_n$, we cover $V_{n,S,j}$ by $N_{n,S,j}$
$L_2$-balls $V_{n,j,S,i}$ of radius
$\ft12jM r_n$ and centre points $\tht_{j,S,i}$.
Comparing the problem of covering $V_{n,j}$ with that of
covering $V_{n,1}$, one realizes that $N_{n,S,j}\leq
N_{n,S}:=N_{n,S,1}$. 

Fix $n\geq1$. Due to lemma~\ref{lem:testineq}, for any test
sequences $\phi_{n,j,S,i}$,
\[
  \begin{split}
  &P_n^{\Pi|B_n} \Pi(V_n|\samplen)
  \leq \sum_{j\geq1}\sum_{S\subset I_n}\sum_{i=1}^{N_{n,S,j}}
    P_n^{\Pi|B_n}\Pi(V_{n,j,S,i}|\samplen)\\
  &\,\leq \frac1{\Pi(B_n)}
    \sum_{j\geq1}\sum_{S\subset I_n}\sum_{i=1}^{N_{n,S,j}}\\
  &\qquad\Bigl( \int_{B_n} P_{\tht,n}\phi_{n,j,S,i}\,d\Pi(\tht)
      + \int_{V_{n,j,S,i}} P_{\tht,n}(1-\phi_{n,j,S,i})\,d\Pi(\tht)\Bigr)\\
  &\,\leq \sum_{p=0}^{Ap_n}\Bigl(\genfrac{}{}{0pt}{}{n}{p}\Bigr)\sum_{j\geq1}
      \,N_{n,S}\,\frac{a_n(j)}{b_n},
  \end{split}
\]
where $b_n:=\Pi(B_n)$, $a_n(j)
:= \max_{S\subset I_n,1\leq i\leq N_{n,S,j}} a_n(j,S,i)$
and,
\[
  \begin{split}
  \frac{a_n(j,S,i)}{b_n} &= \int P_{\tht,n}\phi_{n,j,S,i}\,d\Pi(\tht|B_n)\\
    &\qquad+ \frac{\Pi(V_{n,j,S,i})}{\Pi(B_n)}
        \int P_{\tht,n}(1-\phi_{n,j,S,i})\,d\Pi(\tht|V_{n,j,S,i})\\
    &\leq \sup_{\tht^n\in B_n} P_{\tht,n}\phi_{n,j,S,i}
    + \frac{\Pi(V_{n,j,S,i})}{\Pi(B_n)} 
        \sup_{\tht^n\in V_{n,j,S,i}}P_{\tht,n}(1-\phi_{n,j,S,i}).
  \end{split}
\]
A standard argument (see lemma~5.1 in \cite{Castillo12}) shows that
there exists a test $\phi_{n,j,S,i}$ such that,
\[
  \begin{split}
  P_{0,n}\phi_{n,j,S,i}
    + \frac{\Pi(V_{n,j,S,i})}{\Pi(B_n)} 
      &\sup_{\tht^n\in V_{n,j,S,i}}P_{\tht,n}(1-\phi_{n,j,S,i})\\
  &\leq 2\sqrt\frac{\Pi(V_{n,j,S,i})}{\Pi(B_n)}
    e^{-\ft1{128}j^2M^2\,p_n\log(n/p_n)}
  \end{split}
\]
Note that for every measurable $0\leq\phi\leq1$,
Cauchy's inequality implies that, for all $\tht,\tht'\in\RR^n$
\begin{equation}
  \label{eq:CS}
  P_{\tht,n} \phi \leq \bigl( P_{\tht',n}\phi^2 \bigr)^{1/2}
    \bigl( P_{\tht',n}(dP_{\tht,n}/dP_{\tht',n})^2\bigr)^{1/2}
  \leq \bigl( P_{\tht',n}\phi \bigr)^{1/2}
    e^{\ft12\|\tht-\tht'\|_{2,n}^2}
\end{equation}
We use this to generalise the first term in the above display
to the test uniform over $B_n$ at the
expense of an extra factor, that is,
\[
  \begin{split}
   \sup_{\tht^n\in B_n} P_{\tht,n}\phi_{n,j,S,i}
      + \frac{\Pi(V_{n,j,S,i})}{\Pi(B_n)} 
        &\sup_{\tht^n\in V_{n,j,S,i}}P_{\tht,n}(1-\phi_{n,j,S,i})\\
   &\leq 2\sqrt\frac{\Pi(V_{n,j,S,i})}{\Pi(B_n)}
     e^{-\ft1{256}j^2M^2\,r_n^2 + \ft1{2}d^2\,r_n^2}
  \end{split}
\]

In what appears to be one of the essential (and technically
very demanding) points of \cite{Castillo12},
the proofs of the lemma~5.4 (only after the first line)
and of proposition~5.1 show that
there exists a constant $K>0$ such that,
\[
  \sqrt\frac{\Pi(V_{n,j,S,i})}{\Pi(B_n)} \leq e^{Kr_n^2},
\]
if $G$ has a Lebesgue density $g:\RR\to\RR$ such that there exists a
constant $c>0$ such that
$|\log g(\tht)-\log g(\tht')| \leq c(1+|\tht-\tht'|)$ for all
$\tht,\tht'\in\RR$. This allows for demonstration that
(see the final argument in the proof of proposition~5.1 in
\cite{Castillo12}) if we choose $M>0$ large enough, there
exists a constant $K'>0$ such that for large enough $n$,
\[
  P_n^{\Pi|B_n} \Pi(V_n|\samplen) \leq e^{-K'r_n^2}.
\]
Remote contiguity follows from (\ref{eq:CS}): fix some $n\geq1$
and note that for any $\tht^n\in B_n$,
\[
  (P_{0,n}\phi)^2\leq e^{d^2\,r_n^2}P_{\tht,n}\phi.
\]
Integrating with respect to $\Pi(\cdot|B_n)$ on both sides
shows that,
\[
  P_{0,n}\phi \leq e^{\ft{d^2}2\,r_n} (P_n^{\Pi|B_n}\phi)^{1/2},
\]
so that $P_{0,n}\contig e^{d^2\,r_n^2}P_n^{\Pi|B_n}$.
So if we choose $d^2<K'$, remote contiguity guarantees that
$P_{0,n}\Pi(V_n|\samplen)\to0$.
\end{example}

\subsection{Goodness-of-fit Bayes factors for random walks}

Consider the asymptotic consistency of goodness-of-fit
tests for the transition kernel of a Markov chain with
posterior odds or Bayes factors. Bayesian analyses
of Markov chains on a finite state space are found in
\cite{Strelioff07} and references therein. Consistency
results \cf\ \cite{Walker04} for random walk data are
found in \cite{Ghosal04}. Large-deviation
results for posterior distributions are derived in
\cite{Papangelou96,Eichelsbacher02}. The examples below
are based on ergodicity for remote contiguity and Hoeffding's
inequality for uniformly ergodic Markov chains \cite{Meyn93,Glynn02}
to construct suitable tests. We first prove the analogue of
Schwartz's construction in the case of an ergodic random
walk.

Let $(S,\scrS)$ denote a measurable state space for a
discrete-time, stationary Markov process $P$ describing a random
walk $X^n=\{X_i\in S:0\leq i\leq n\}$ of length
$n\geq1$ (conditional on a starting position $X_0$). The
chain has a Markov transition kernel 
$P(\cdot|\cdot):\scrS\times S\to[0,1]$
that describes $X_i|X_{i-1}$ for all $i\geq1$.

Led by Pearson's approach to goodness-of-fit testing, we
choose a finite partition
$\al=\{A_1,\ldots,A_N\}$ of $S$ and `bin the data' in the
sense that we switch to a new process $Z^n$ taking values
in the finite state space $S_\al=\{e_j:1\leq j\leq N\}$
(where $e_j$ denotes the $j$-th standard basis vector in $\RR^n$),
defined by $Z^n=\{Z_i\in S_\al:0\leq i\leq n\}$, with
$Z_i=(1\{X_i\in A_1\},\ldots,1\{X_i\in A_N\})$. The
process $Z^n$ forms a stationary Markov chain on $S_\al$
with distribution $P_{\al,n}$. The model is parametrized in
terms of the convex set $\Tht$ of $N\times N$ Markov transition
matrices $p_\al$ on the finite state space $S_\al$,
\begin{equation}
  \label{eq:alphaproj}
  p_{\al}(k|l)=P_{\al,n}(Z_i=e_k|Z_{i-1}=e_l) = P(X_i\in A_k| X_{i-1}\in A_l),
\end{equation}
for all $0\leq i\leq n$ and $1\leq k,l\leq N$. We assume that
$P_{\al,n}$ is ergodic with equilibrium distribution that we
denote by $\pi_\al$, and $\pi_\al(k):=\pi_\al(Z=k)$.
We are interested in Bayes factors for goodness-of-fit type
questions, given a parameter space consisting of transition
matrices. 
\begin{example}
\label{ex:KLergodic}
Assume that the true transition kernel $P_0$ gives rise to
a matrix $p_0\in\Tht$ that generates an ergodic Markov chain $Z^n$.
Denote the true distribution of $Z^n$ by $P_{0,n}$ and the equilibrium
distribution by $\pi_0$ (with $\pi_0(k):=\pi_0(Z=k)$). For given
$\ep>0$, define,
\[
  B' = \Bigl\{ p_\al\in \Tht:
    \sum_{k,l=1}^N -p_0(l|k)\pi_0(k)\log \frac{p_\al(l|k)}{p_0(l|k)}
    < \ep^2 \Bigr\}.
\]
Assume that $\Pi(B')>0$. According to the ergodic theorem, for every
$p_\al\in B'$,
\[
  \frac1n\sum_{i=1}^n \log \frac{p_\al(Z_i|Z_{i-1})}{p_0(Z_i|Z_{i-1})}
  \convas{P_{0,n}}\sum_{k,l=1}^N p_0(l|k)\pi_0(k)\log \frac{p(l|k)}{p_0(l|k)},
\]
(compare with the rate-function in the large-deviation results in
\cite{Papangelou96,Eichelsbacher02}) so that, for large enough $n$,
\[
  \frac{dP_{\al,n}}{dP_{0,n}}(Z^n) = \prod_{i=1}^n
  \frac{p_\al(Z_i|Z_{i-1})}{p_0(Z_i|Z_{i-1})}
    \geq e^{-\ft{n}{2}\ep^2},
\]
$P_{0,n}$-almost-surely.
Just like in Schwartz's proof \cite{Schwartz65}, in
proposition~\ref{prop:rcfiniteX} and in example~\ref{ex:regression},
the assumption $\Pi(B')>0$ and Fatou's lemma imply remote contiguity
because,
\[
  P_{0,n}\Bigl( \int \frac{dP_{\al,n}}{dP_{0,n}}(Z^n)\,d\Pi(p_\al|B')
    < e^{-\ft{n}{2}\ep^2} \Bigr) \to 0.
\]
So lemma~\ref{lem:rcfirstlemma} says that
$P_{0,n}\contig \exp(\ft{n}{2}\ep^2)P_n^{\Pi|B'}$.
\end{example}
However, exponential remote contiguity will turn out not to be
enough for goodness-of-fit tests below, unless we impose
stringent model conditions. Instead, we shall resort to local
asymptotic normality for a sharper result.
\begin{example}
\label{ex:gofmarkov}
We formulate goodness-of-fit
hypotheses in terms of the joint distribution for two
consecutive steps in the random walk.
Like Pearson, we fix some such distribution $P_0$ and
consider hypotheses based on differences of `bin probabilities'
$p_\al(k,l)=p_\al(k|l)\pi_\al(l)$,
\begin{equation}
  \label{eq:H0H1}
  \begin{split}
    H_0:\max_{1\leq k,l\leq N} \bigl|p_{\al}(k,l)-p_{0}(k,l)\bigr| &< \ep,\\
    H_1:\max_{1\leq k,l\leq N} \bigl|p_{\al}(k,l)-p_{0}(k,l)\bigr| &\geq \ep,
  \end{split}
\end{equation}
for some fixed $\ep>0$. The sets $B$ and $V$ are defined as
the sets of transition matrices $p_\al\in\Tht$ that satisfy
hypotheses $H_0$ and $H_1$ respectively. We assume that
the prior is chosen such that $\Pi(B)>0$ and $\Pi(V)>0$.

Endowed with some matrix norm, $\Tht$ is compact and a
Borel prior on $\Tht$ can be defined in various ways. For example,
we may assign the vector $( p_\al(\cdot|1), \ldots, p_\al(\cdot|N))$  
a product of Dirichlet distributions. Conjugacy applies and
the posterior for $p_\al$ is again a product of Dirichlet
distributions \cite{Strelioff07}. For an alternative family of
priors, consider the set $\scrE$ of $N^N$ $N\times N$-matrices $E$ that
have standard basis vectors $e_k$ in $\RR^N$ as columns. Each
$E\in\scrE$ is a deterministic Markov transition matrix on $S_\al$
and $\scrE$ is the extremal set of the polyhedral set $\Tht$.
According to Choquet's theorem, every transition matrix $p_\al$
can then be written in the form,
\begin{equation}
  \label{eq:choquet}
  p_\al = \sum_{E\in \scrE} \lambda_E\,E,
\end{equation}
for a (non-unique) combination of
$\lambda_\scrE:=\{\lambda_E:E\in\scrE\}$
such that $\lambda_E\geq0$, $\sum_\scrE\lambda_E=1$. If $\lambda_E>0$
for all $E\in\scrE$, the resulting Markov chain is ergodic
and we denote the corresponding distributions for $Z^n$ by $P_{\al,n}$.
Any Borel prior $\Pi'$ (\eg\ a Dirichlet distribution) on the simplex
$S_{N^N}$ in $\RR^{N^N}$ is a prior
for $\lambda_{\scrE}$ and induces a Borel prior $\Pi$ on
$\Tht$. Note that all non-ergodic transition matrices lie in the
boundary $\partial\Tht$, so if we choose $\Pi'$ such that
$\Pi(\mathring{\Tht})=1$, ergodicity may be assumed
in all prior-almost-sure arguments. This is true for any $\Pi'$
that is absolutely continuous with respect to the ($N^N-1$-dimensional)
Lebesgue measure on $S_{N^N}$ (for example when we choose
$\Pi'$ equal to a Dirichlet distribution). Note that if the
associated density is continuous and strictly positive,
$\Pi(B)>0$ and $\Pi(V)>0$. 

We intend to use theorem~\ref{thm:bayesfactor} with $B$ and $V$ defined by
$H_0$ and $H_1$, so we first demonstrate that a Bayesian
test sequence for $B$ versus $V$ exists, based on a version
of Hoeffding's inequality valid for random walks \cite{Glynn02}.
First, define, for given $0<\lambda_n\leq N^{-N}$ such that
$\lambda_n\downarrow0$,
\[
  S'_n:=\bigl\{ \lambda_{\scrE}\in S^{N^N}:
    \lambda_E\geq\lambda_n/N^{N-1}, \text{for all }{E\in\scrE} \bigr\},
\]
and denote the image of $S'_n$ under (\ref{eq:choquet}) by
$S_n$. Note that if $\Pi(\partial\Tht)=0$, then
$\pi_{S,n}:=\Pi(\Tht\setminus S_n)\to0$.

Now fix $n\geq1$ for the moment. Recalling the nature of the matrices
$E$, we see that for every $1\leq k,l\leq N$, $p_\al(k|l)$ as in
equation (\ref{eq:choquet}) is greater than or equal to $\lambda_n$.
Consequently, the corresponding Markov chain satisfies
condition~(A.1) of Glynn and Ormoneit \cite{Glynn02}
(closely related to the notion of uniform ergodicity
\cite{Meyn93}): starting in any point $X_0$ under a
transition from $S_n$, the probability
that $X_1$ lies in $A\subset S_\al$ is greater than or equal to
$\lambda_n\,\phi(A)$, where $\phi$ is the uniform probability
measure on $S_\al$. This mixing condition enables a
version of Hoeffding's inequality (see theorem~2 in
\cite{Glynn02}): for any $\lambda_\scrE\in S'_n$ and
$1\leq k,l\leq N$,
the transition matrix of equation~(\ref{eq:choquet})
is such that, with $\hat{p}_n(k,l)=n^{-1}\sum_i1\{Z_i=k,Z_{i-1}=l\}$,
\begin{equation}
  \label{eq:hoeff}
  P_{\al,n}\bigl( \hat{p}_n(k,l) - p_\al(k,l) \geq\delta \bigr)
    \leq \exp\Bigl( -\frac{\lambda_n^2(n\delta-2\lambda_n^{-1})^2}{2n} \Bigr).
\end{equation}
Now define for a given sequence
$\delta_n>0$ with $\delta_n\downarrow0$ and all $n\geq1,1\leq k,l\leq N$,
\[
  \begin{split}
  B_n&=\{p_\al\in\Tht: 
    \max_{k,l}\bigl|p_{\al}(k,l)-p_{0}(k,l)\bigr| < \ep-\delta_n\},\\
  V_{k,l} &= \{p_\al\in\Tht: 
    \bigl|p_{\al}(k,l)-p_{0}(k,l)\bigr| \geq \ep\},\\
  V_{+,k,l,n}&=\{p_\al\in\Tht: 
    p_{\al}(k,l)-p_{0}(k,l) \geq \ep+\delta_n\},\\
  V_{-,k,l,n}&=\{p_\al\in\Tht: 
    p_{\al}(k,l)-p_{0}(k,l) \leq -\ep-\delta_n\}.
  \end{split}
\]
Note that if $\Pi'$ is absolutely continuous with respect to the
Lebesgue measure on $S^{N^N}$, then $\pi_{B,n}:=\Pi(B\setminus B_n)\to0$
and $\pi_{n,k,l}:=\Pi(V_{k,l}\setminus (V_{+,k,l,n}\cup V_{-,k,l,n}))\to0$.

If we define the test $\phi_{+,k,l,n}(Z^n)
=1\{ \hat{p}_n(k,l) - p_0(k,l) \geq\ep\}$, 
then for any $p_\al\in B_n\cap S_n$,
\[
  \begin{split}
  P_{\al,n}\phi_{+,k,l,n}(Z^n)
    &\leq P_{\al,n}\bigl( \hat{p}_n(k,l) - p_\al(k,l)
      \geq\delta_n \bigr)\\
    &\leq \exp\Bigl(
      -\frac{\lambda_n^2(n\delta_n-2\lambda_n^{-1})^2}{2n} \Bigr).
  \end{split}
\]
If on the other hand, 
$p_\al$ lies in the intersection of $V_{+,n,k,l}$ with $S_n$, we find,
\[
  \begin{split}
  P_{\al,n}(1-\phi_{+,n,k,l}(Z^n))
    &= P_{\al,n}\bigl( \hat{p}_n(k,l) - p_\al(k,l)
      < -\delta_n \bigr)\\
    &\leq \exp\Bigl(
      -\frac{\lambda_n^2(n\delta_n-2\lambda_n^{-1})^2}{2n} \Bigr).
  \end{split}
\]
Choosing the sequences $\delta_n$ and $\lambda_n$ such that
$n\delta_n^2\lambda_n^2\to\infty$, we also have
$\lambda_n^{-1}=o(n\delta_n)$, so the exponent on the right is smaller
than or equal to $-\ft18n\lambda_n^2\delta_n^2$.

So if we define
$\phi_n(Z^n)=\max_{k,l}\{\phi_{-,k,l,n}(Z^n),\phi_{+,k,l,n}(Z^n)\}$,
\[
  \begin{split}
    \int_{B} &P_{\al,n}\phi_n\,d\Pi(p_\al)
      + \int_{V} Q_{\al,n}(1-\phi_n)\,d\Pi(q_\al)\\
      &\leq \int_{B\cap S_n} P_{\al,n}\phi_n\,d\Pi(p_\al)
      + \int_{V\cap S_n} Q_{\al,n}(1-\phi_n)\,d\Pi(q_\al)
      +\Pi(\Tht\setminus S_n)\\
      &\leq \int_{B} \sum_{k,l=1}^N 
        P_{\al,n}(\phi_{-,k,l,n}+\phi_{+,k,l,n})\,d\Pi(p_\al)\\
      &\quad + \sum_{k,l=1}^N \Bigl(\int_{V_{-,k,l}}
          Q_{\al,n}(1-\phi_{-,k,l,n})\,d\Pi(q_\al)\\
      &\quad+ \int_{V_{+,k,l}}
          Q_{\al,n}(1-\phi_{+,k,l,n})\,d\Pi(q_\al)\Bigr)\\
      &\quad+\sum_{k,l=1}^N
        \Pi\bigr(V_{n,k,l}\setminus(V_{+,n,k,l}\cup V_{+,n,k,l})\bigl)
        +\Pi(\Tht\setminus S_n)+\Pi(B\setminus B_n)\\
      &\leq 2N^2 e^{-\ft18n\lambda_n^2\delta_n^2} + \pi_{B,n}+\pi_{S,n}
        +\sum_{k,l=1}^N \pi_{n,k,l}.
  \end{split}
\]
So if we choose a prior $\Pi'$ on $S^{N^N}$ that is absolutely continuous
with respect to Lebesgue measure, then $(\phi_n)$ defines a
Bayesian test sequence for $B$ versus $V$.

Because we have not imposed control over the rates at which the terms
on the \rhs\ go to zero, remote contiguity at exponential rates is
not good enough. Even if we would restrict supports of a sequence of
priors such that $\pi_{B,N}=\pi_{S,n}=\pi_{n,k,l}=0$, the first term
on the \rhs\ is sub-exponential. To obtain a rate sharp enough,
we note that the chain $Z^n$ is positive recurrent, which
guarantees that the dependence $p_\al\to{dP_{\al,n}}/{dP_{0,n}}$
is locally asymptotically normal \cite{Hopfner90,Gobet02}.
According to lemma~\ref{lem:rcLAN}, this implies that local
prior predictive distributions based on $n^{-1/2}$-neighbourhoods
of $p_0$ in $\Tht$ are $c_n$-remotely contiguous to $P_{0,n}$ for
\emph{any} rate $c_n$, if the prior has full support. If we require
that the prior density $\pi'$ with respect to Lebesgue measure on $S^{N^N}$ 
is continuous and strictly positive, then we see that there
exists a constant $\pi>0$ such that $\pi'(\lambda)\geq\pi$ for
all $\lambda\in S^{N^N}$, so that for every $n^{-1/2}$-neighbourhood
$B_n$ of $p_0$, there exists a $K>0$ such that
$\Pi(B_n)\geq b_n:=K\,n^{-{N^N}/2}$. Although local asymptotic
normality guarantees remote contiguity at arbitrary rate, we still have
to make sure that $c_n\to0$ in lemma~\ref{lem:rcLAN}, \ie\ that
$a_n=o(b_n)$. Then the remark directly after theorem~\ref{thm:bayesfactor}
shows that condition~{\it (ii)} of said theorem is satisfied.

The above leads to the following
conclusion concerning goodness-of-fit testing \cf\ (\ref{eq:H0H1}).
\begin{proposition}
\label{prop:gofmarkov}
Let $X^n$ be a stationary, discrete time Markov chain on a
measurable state space $(S,\scrS)$. Choose a finite, measurable
partition $\al$ of $S$ such that the Markov chain $Z^n$
is ergodic. Choose a prior $\Pi'$ on
$S^{N^N}$ absolutely continuous with respect to
Lebesgue measure with a continuous density that is
everywhere strictly positive. Assume that,
\begin{itemize}
  \item[(i)]  $n\lambda_n^2\delta_n^2/\log(n)\to\infty,$
  \item[(ii)] $\Pi(B\setminus B_n),\Pi(\Tht\setminus S_n)=o(n^{-({N^N}/2)})$,
  \item[(iii)] $\max_{k,l}\Pi(V_{k,l}\setminus (V_{+,k,l,n}\cup V_{-,k,l,n}))
  =o(n^{-({N^N}/2)})$.
\end{itemize}
Then for any choice of
$\ep>0$, the Bayes factors $F_n$ are consistent for $H_0$ versus
$H_1$.
\end{proposition}
To guarantee ergodicity of $Z^n$ one may use an empirical device,
\ie\ we may use an independent, finite-length realization
of the random walk $X^n$ to find
a partition $\al$ such that for all $1\leq k,l\leq N$, we observe
some $m$-step transition from $l$ to $k$. An interesting generalisation
concerns
a hypothesized Markov transition kernel $P_0$ for the process $X^n$
and partitions $\al_n$ (with projections $p_{0,\al_n}$ as in
(\ref{eq:alphaproj})), chosen such that $\al_{n+1}$ refines $\al_n$ for
all $n\geq1$. Bayes factors then test a sequence of pairs of hypotheses
(\ref{eq:H0H1}) centred on the $p_{0,\al_n}$. The arguments leading
to proposition~\ref{prop:gofmarkov} do not require modification
and the rate of growth $N_n$ comes into the conditions of
proposition~\ref{prop:gofmarkov}. \closebox
\end{example}
Example~\ref{ex:gofmarkov} demonstrates the enhancement of the role
of the prior as intended by the remark that closes the subsection
on the existence of Bayesian test sequences in section~\ref{sec:post}:
where testing power is relatively weak, prior mass should be scarce
to compensate and where testing power is strong, prior mass
should be plentiful. A random walk for which mixing does not occur
quickly enough does not give rise to (\ref{eq:hoeff}) and alternatives
for which separation decreases too fast lose testing power, so the
difference sets of proposition~\ref{prop:gofmarkov} are the
hard-to-test parts of the parameter space and conditions~{\it(ii)--(iii)}
formulate how scarce prior mass in these parts has to be.

\subsection{Finite sample spaces and the tailfree case}

\begin{example}
\label{ex:rcfiniteX}
Consider the situation where we observe an \iid\ sample
of random variables $X_1,X_2,\ldots$ taking values in a space
$\scrX_N$ of \emph{finite order} $N$. Writing $\scrX_N$ as
the set of integers $\{1,\ldots,N\}$, we note that the space
$M$ of all probability measures $P$ on $(\scrX_N,2^{\scrX_N})$ with
the total-variational metric $(P,Q)\mapsto\|P-Q\|$ is in
isometric correspondence with the simplex,
\[
  S_N = \bigl\{ p=(p(1),\ldots,p(N)):\min_kp(k)\geq0,
  \Sigma_i\, p(i)=1 \bigr\},
\]
with the metric $(p,q)\mapsto\|p-q\|=\Sigma_k\,|p(k)-q(k)|$ it
inherits from $\RR^N$ with the $L_1$-norm, when $k\mapsto p(k)$
is the density of $P\in M$ with respect to the
counting measure. We also define $M'=\{P\in M: P(\{k\})>0,
1\leq k\leq N\}\subset M$ (and $R_N=\{p\in S_N:p(k)>0,
1\leq k\leq N\}\subset S_N$). 
\begin{proposition}
\label{prop:rcfiniteX}
If the data is an \iid\ sample of $\scrX_N$-valued random variables,
then for any $n\geq1$, any Borel prior
$\Pi:\scrG\to[0,1]$ of full support on $M$, any $P_0\in M$ and any
ball $B$ around $P_0$, there exists an $\ep'>0$ such that,
\begin{equation}
  \label{eq:rcfiniteX}
  P_0^n\ctg e^{\ft12n\ep^2} P_n^{\Pi|B},
\end{equation}
for all $0<\ep<\ep'$.
\end{proposition}
\begin{proof}
By the inequality $\|P-Q\|\leq-P\log(dQ/dP)$, the
ball $B$ around $P_0$ contains all sets of the form
$K(\ep)=\{P\in M':-P_0\log(dP/dP_0)<\ep\}$, for some $\ep'>0$ and all
$0<\ep<\ep'$. Fix such an $\ep$. Because the mapping
$P\mapsto -P_0\log(dP/dP_0)$ is continuous on $M'$, there exists an
open neighbourhood $U$ of $P_0$ in $M$ such that $U\cap M'\subset K(\ep)$.
Since both $M'$ and $U$ are open and $\Pi$ has full support,
$\Pi(K(\ep))\geq\Pi(U\cap M')>0$. With the help of
example~\ref{ex:KLclose}, we see that for every $P\in K(\ep)$,
\[
  e^{\ft12n\ep^2}\frac{dP^n}{dP_0^n}(\samplen)\geq 1,
\]
for large enough $n$, $P_0$-almost-surely. Fatou's lemma again
confirms
condition~{\it (ii)} of
lemma~\ref{lem:rcfirstlemma} is satisfied. Conclude that
assertion~(\ref{eq:rcfiniteX}) holds.
\end{proof}
\end{example}

\begin{example}
\label{ex:finiteX}
We continue with the situation where we observe an \iid\ sample
of random variables $X_1,X_2,\ldots$ taking values in a space
$\scrX_N$ of finite order $N$. For given $\delta>0$, consider
the hypotheses,
\[
  B=\{ P\in M:\|P-P_0\|<\delta \},\,\,
  V=\{ Q\in M:\|Q-P_0\|>2\delta \}.
\]
Noting that $M$ is compact (or with the help of the simplex
representation $S_N$) one sees that entropy numbers of $M$
are bounded, so the construction of example~\ref{ex:entropy}
shows that uniform tests of exponential power $e^{-nD}$
(for some $D>0$) exist for $B$ versus $V$. 
Application of proposition~\ref{prop:rcfiniteX} shows that
the choice for an $0<\ep<\ep'$ small enough, guarantees
that $\Pi(V|\samplen)$ goes to zero in
$P_0^n$-probability. Conclude that the posterior resulting
from a prior $\Pi$ of full support on $M$ is consistent in
total variation.
\end{example}

\begin{example}
\label{ex:tailfree}
With general reference to Ferguson (1973) \cite{Ferguson73},
one way to construct non-parametric priors concerns a
refining sequence of finite, Borel measurable partitions
of a Polish sample space, say $\scrX=\RR$: to define a `random
distribution' $P$ on $\scrX$, we specify for each such
partition $\alpha=\{A_1,\ldots,A_N\}$, a Borel prior
$\Pi_\alpha$ on $S_N$, identifying $(p_1,\ldots,p_N)$
with the `random variables' $(P(A_1),\ldots,P(A_N))$.
Kolmogorov existence of the stochastic process
describing all $P(A)$ in a coupled way subjects these
$\Pi_\alpha$ to consistency requirements expressing
that if $A_1,A_2$ partition $A$, then $P(A_1)+P(A_2)$
must have the same distribution as $P(A)$. If the
partitions refine appropriately, the resulting process
describes a probability measure $\Pi$ on the space of Borel
probability measures on $\scrX$, \ie\ a `random distribution'
on $\scrX$. Well-known examples of priors that
can be constructed in this way are the Dirichlet process
prior (for which a so-called base-measure $\mu$ supplies
appropriate parameters for Dirichlet distributions
$\Pi_{\alpha}$, see \cite{Ferguson73}) and Polya Tree
prior (for detailed explanations, see, for example,
\cite{Ghosh03}).

A special class of priors constructed in this way are
the so-called \emph{tailfree} priors. The process
prior associated with a family of $\Pi_\alpha$ like
above is said to be \emph{tailfree}, if for all
$\alpha,\beta$ such that $\beta=\{B_1,\ldots,B_M\}$
refines $\alpha=\{A_1,\ldots,A_N\}$, the following holds:
for all $1\leq k\leq N$, $(P(B_{l_1}|A_k),\ldots,P(B_{l_L(k)}|A_k))$
(where the sets $B_{l_1},\ldots,B_{l_L(k)}\in\beta$
partition $A_k$) is independent of $(P(A_1),\ldots,P(A_N))$.
Although somewhat technical, explicit
control of the choice for the $\Pi_\alpha$ render
the property quite feasible in examples. 

Fix a finite, measurable partition $\alpha=\{A_1,\ldots,A_N\}$.
For every $n\geq1$, denote by $\sigma_{\al,n}$ the $\sigma$-algebra
$\sigma(\alpha^n)\subset\scrB^n$, generated by products of
the form $A_{i_1}\times \dots\times A_{i_n}\subset\scrX^n$, with
$1\leq i_1,\ldots,i_n\leq N$. Identify $\scrX_N$ with the
collection $\{e_1,\ldots,e_N\}\subset\RR^N$ and define the projection
$\varphi_\alpha:\scrX\mapsto\scrX_N$ by,
\[
  \varphi_\alpha(x)=\bigl(1\{x\in A_1\},\ldots,1\{x\in A_N\}\bigr).
\]
We view $\scrX_N$ (respectively $\scrX_N^n$) as a probability space, with
$\sigma$-algebra $\sigma_N$ equal to the power set (respectively
$\sigma_{N,n}$, the power set of $\scrX_N^n$) and probability
measures denoted $P_\al:\sigma_N\to[0,1]$ that we identify with
elements of $S_N$. Denoting the space of all Borel probability
measures on $\scrX$ by $M^1(\scrX)$, we also define
$\varphi_{\ast\alpha}:M^1(\scrX)\to S_N$,
\[
  \varphi_{\ast\alpha}(P)=\bigl(P(A_1),\ldots,P(A_N)\bigr),
\]
which maps $P$ to its restriction to $\sigma_{\alpha,1}$, a
probability measure on $\scrX_N$. Under the projection $\phi_\alpha$,
any Borel-measurable random variable $X$ taking values in
$\scrX$ distributed $P\in M^1(\scrX)$ is mapped to a random
variable $Z_\al=\varphi_\alpha(X)$ that takes values in $\scrX_N$
(distributed $P_{\alpha}=\varphi_{\ast\alpha}(P)$). We also
define 
$Z_\alpha^n=(\varphi_\alpha(X_1),\ldots, \varphi_\alpha(X_n))$,
for all $n\geq1$. 

Let $\Pi_\alpha$ denote a Borel prior on $S_N$. The posterior on
$S_N$ is then a Borel measure denoted $\Pi_{\alpha}(\cdot|Z_\alpha^n)$,
which satisfies, for all $A\in\sigma_{N,n}$ and any Borel set $V$ in
$S_N$,
\[
  \int_A \Pi_\alpha(V|Z_\alpha^n)\,dP^{\Pi_\alpha}_n
    = \int_V P_\alpha^n(A)\,d\Pi_{\alpha}(P_\alpha),
\]
by definition of the posterior. In the model for the
original \iid\ sample $\samplen$,
Bayes's rule takes the form, for all $A'\in\scrB_n$ and all
Borel sets $V'$ in $M^1(\scrX)$,
\[
  \int_{A'} \Pi(V'|\samplen)\,dP^{\Pi}_n
    = \int_{V'} P^n(A')\,d\Pi(P),
\]
defining the posterior for $P$. Now specify that $V'$ is the
pre-image $\varphi_{\ast\alpha}^{-1}(V)$ of a Borel measurable $V$
in $S_N$:
as a consequence of tailfreeness, the data-dependence of the
posterior for such a $V'$, $\samplen\mapsto\Pi(V'|\samplen)$,
is measurable with respect to 
$\sigma_{\alpha,n}$ (see Freedman (1965) \cite{Freedman65} 
or Ghosh (2003) \cite{Ghosh03}). So there exists a function
$g_n:\scrX_N^n\to[0,1]$ such that,
\[
  \Pi(V'|\samplen=x^n)=g_n(\varphi_\alpha(x_1),\ldots,\varphi_\alpha(x_n)),
\]
for $P_n^\Pi$-almost-all $x^n\in\scrX^n$. Then, for given $A'\in\sigma_{\al,n}$
(with corresponding $A\in\sigma_{N,n}$),
\[
  \begin{split}
  \int_{A'}&\Pi(V'|\samplen)\,dP^{\Pi}_n
    = \int P^n(1_{A'}(\samplen)\,\Pi(V'|\samplen))\,d\Pi(P)\\
    &= \int P_\alpha^n\bigl(1_{A}(Z_\alpha^n)
      \,g_n(Z_\alpha^n)\bigr)\,d\Pi_{\alpha}(P_\alpha)
      = \int_{A} g_n(Z_\alpha^n)\,dP^{\Pi_\alpha}_n,
  \end{split}
\]
while also,
\[
  \int_{V'}P^n(A')\,d\Pi(P) = \int_{V}P_{\alpha}^n(A)\,d\Pi_\alpha(P_\alpha).
\]
This shows that $Z_\alpha^n\mapsto g_n(Z_\alpha^n)$ is a version of
the posterior $\Pi_\alpha(\,\cdot\,|Z_\alpha^n)$. In other words, we
can write $\Pi(V'|\samplen)=\Pi_\al(V|\phi_\alpha(\samplen))
=\Pi_\alpha(V|Z_\alpha^n)$, $P_n^\Pi$-almost-surely.

Denote the true distribution of a single observation from $X^n$
by $P_0$. For any $V'$ of the form $\varphi_{\ast\alpha}^{-1}(V)$
for some $\al$
and a neighbourhood $V$ of $P_{0,\alpha}=\varphi_{\ast\al}(P_0)$ in $S_N$,
the question whether $\Pi(V'|\samplen)$ converges to one in $P_0$-probability
reduces to the question whether $\Pi(V|Z_{\alpha}^n)$ converges to one
in $P_{0,\alpha}$-probability. Remote
contiguity then only has to hold as in example~\ref{ex:rcfiniteX}.

Another way of saying this is to note directly that,
because $\samplen\mapsto\Pi(V'\,|\samplen)$ is
$\sigma_{\alpha,n}$-measurable, remote contiguity
(as in definition~\ref{def:remctg}) is to be
imposed \emph{only} for $\phi_n:\scrX^n\to[0,1]$ that
are measurable with respect to $\sigma_{\alpha,n}$
(rather than $\scrB^n$) for every
$n\geq1$. That conclusion again reduces the remote contiguity
requirement necessary for the consistency of
the posterior for the parameter $(P(A_1),\ldots P(A_N))$
to that of a finite sample space, as in
example~\ref{ex:rcfiniteX}. Full support of the prior
$\Pi_\alpha$ then guarantees remote contiguity
for exponential rates as required in condition~{\it(ii)}
of theorem~\ref{thm:consistency}. In the case of the
Dirichlet process prior, full support of the base measure
$\mu$ implies full support for all $\Pi_\alpha$, if we
restrict attention to partitions $\alpha=(A_1,\ldots,A_N)$
such that $\mu(A_i)>0$ for all $1\leq i\leq N$. (Particularly,
we require $P_0\ll\mu$ for consistent estimation.)

Uniform tests of exponential power
for weak neighbourhoods complete the proof
that tailfree priors lead to weakly consistent
posterior distributions: (norm) consistency of
$\Pi_\alpha(\,\cdot\,|Z_\alpha^n)$ for all $\alpha$
guarantees (weak $\scrT_1$-)consistency
of $\Pi(\,\cdot\,|\samplen)$, in this proof based on remote
contiguity and theorem~\ref{thm:consistency}.

\end{example}

\subsection{Credible/confidence sets in metric spaces}

When enlarging credible sets to confidence sets using a collection
of subsets $B$ as in definition~\ref{def:confcred}, measurability
of confidence sets is guaranteed if $B(\tht)$ is
\emph{open} in $\Tht$ for all $\tht\in\Tht$.
\begin{example}
\label{ex:uniform}
Let $\scrG$ be the Borel $\sigma$-algebra for a uniform topology on
$\Tht$, like the weak and metric topologies of appendix~\ref{sec:defs}.
Let $W$ denote a symmetric entourage and, for every $\tht\in\Tht$, define
$B(\tht)=\{\tht'\in\Tht:(\tht,\tht')\in W\}$, a neighbourhood of $\tht$.
Let $D$ denote any credible set. 
A confidence set associated with $D$ under $B$ is any set $C'$
such that the complement of $D$ contains the $W$-enlargement of the
complement of $C'$. Equivalently (by the symmetry of $W$),
the $W$-enlargement of $D$ does not meet the complement of $C'$.
Then the minimal confidence set $C$ associated with $D$ is the
$W$-enlargement of $D$. If the $B(\tht)$ are all open neighbourhoods
(\eg\ whenever $W$ is a symmetric entourage from a fundamental system
for the uniformity on $\Tht$), the minimal confidence set associated
with $D$ is open. The most common examples include the Hellinger or
total-variational metric uniformities, but weak topologies (like Prohorov's
or $\scrT_n$-topologies) and polar topologies are uniform too.\closebox
\end{example}
\begin{example}
\label{ex:confballs}
To illustrate example~\ref{ex:uniform} with a customary situation, consider
a parameter space $\Tht$ with parametrization $\tht\mapsto P_{\tht}^n$,
to define a model for \iid\ data $\samplen=(X_1,\ldots,X_n)\sim P_{\tht_0}^n$,
for some $\tht_0\in\Tht$. Let $\scrD$ be the class of all pre-images of
Hellinger balls, \ie\ sets $D(\tht,\ep)\subset\Tht$ of the form,
\[
  D(\tht,\ep)=\bigl\{\,\tht'\in\Tht:H(P_{\tht},P_{\tht'})<\ep\,\bigr\},
\]
for any $\tht\in\Tht$ and $\ep>0$. After choice of a Kullback-Leibler
prior $\Pi$ for $\tht$ and calculation of the posteriors, choose $D_n$
equal to the pre-image $D(\hat{\tht}_n,\hat\ep_n)$ of
a (\eg\ the one with the smallest radius, if that exists) Hellinger ball
with credible level $1-o(a_n)$, $a_n=\exp(-n\alpha^2)$ for
some $\alpha>0$. Assume, now, that for some
$0<\ep<\alpha$, the $W$ of example~\ref{ex:uniform} is the Hellinger entourage
$W=\{(\tht,\tht'):H(P_{\tht},P_{\tht'})<\ep\}$. Since Kullback-Leibler
neighbourhoods are contained in Hellinger balls, the sets
$D(\hat{\tht}_n,\hat\ep_n+\ep)$
(associated with $D_n$ under the entourage $W$),
is a sequence of asymptotic confidence sets, provided the prior
satisfies (\ref{eq:KLprior}). If we make $\ep$
vary with $n$, neighbourhoods of the form $B_n$ in
example~\ref{ex:GGV} are contained in Hellinger balls of radius
$\ep_n$, and in that case,
\[
  C_n(\samplen) = D(\hat{\tht}_n,\hat\ep_n+\ep_n),
\]
is a sequence of asymptotic confidence sets, provided that the prior
satisfies (\ref{eq:GGV}).
\closebox
\end{example}


\section{Proofs}
\label{sec:proofs}

In this section of the appendix, proofs from the main text are collected.

\begin{proof} (theorem~\ref{thm:doob})\\
The argument (see, \eg, Doob (1949) \cite{Doob49} or Ghosh
and Ramamoorthi (2003) \cite{Ghosh03}) relies on martingale convergence
and a demonstration of the existence of a measurable
$f:\scrX^{\infty}\to\scrP$ such that $f(X_1,X_2,\ldots)=P$,
$P^\infty$-almost-surely for all $P\in\scrP$
(see also propositions~1 and~2 of section~17.7 in \cite{LeCam86}).
\end{proof}

\begin{proof} (proposition~\ref{lem:testineq})\\
Due to Bayes's Rule (\ref{eq:disintegration}) and monotone convergence,
\[
  \begin{split}
  \int_BP_{\tht}&(1-\phi)\,\Pi(V|X)\,d\Pi(\tht)\\
  &\leq \int(1-\phi)\,\Pi(V|X)\,dP^\Pi= \int_V P_{\tht}(1-\phi)\,d\Pi(\tht).
  \end{split}
\]
Inequality (\ref{eq:testineq}) follows from the fact that $\Pi(V|X)\leq1$.
\end{proof}

\begin{proof} (theorem~\ref{thm:testconsequi})\\
Condition {\it (i)} implies {\it (ii)} by dominated convergence.
Assume {\it (ii)} and note that by lemma~\ref{lem:testineq},
\[
  \int P_{\tht,n}\Pi(V|\samplen)\,d\Pi(\tht|B)\to0.
\]
Assuming that the observations $\samplen$ are coupled and can
be thought of as projections of a random variable
$X\in\scrX^\infty$ with distribution $P_\tht$, martingale convergence in
$L^1(\scrX^\infty\times\Tht)$ (relative to the probability measure
$\Pi^{\ast}$ defined by $\Pi^{\ast}(A\times B)
=\int_B P_{\tht}(A)\,d\Pi(\tht)$ for measurable $A\subset\scrX^{\infty}$
and $B\subset\Tht$), shows there is a measurable
$g:\scrX^\infty\to[0,1]$ such that,
\[
  \int P_{\tht}\bigl| \Pi(V|\samplen)
    -g(X)\bigr|\,d\Pi(\tht|B)\to 0.
\]
So $\int P_{\tht} g(X)\,d\Pi(\tht|B)=0$, implying that $g=0$,
$P_{\tht}$-almost-surely for $\Pi$-almost-all $\tht\in B$. Using
martingale convergence again
(now in $L^\infty(\scrX^\infty\times\Tht)$), conclude
$\Pi(V|\samplen)\to0$, $P_{\tht}$-almost-surely for
$\Pi$-almost-all $\tht\in B$, from which {\it (iii)} follows. 
Choose $\phi(\samplen)=\Pi(V|\samplen)$ to conclude that
{\it (i)} follows from {\it (iii)}.
\end{proof}

\begin{proof} (proposition~\ref{prop:msbtest})\\
Apply \cite{LeCam86}, section~17.1, proposition~1
with the indicator for $V$. See also \cite{Breiman64}.
\end{proof}

\begin{proof} (lemma~\ref{lem:rcfirstlemma})\\
Assume {\it (i)}. Let 
$\phi_n:\scrX_n\rightarrow[0,1]$ be given and assume that
$P_n\phi_n=o(a_n)$. By Markov's inequality, for every $\ep>0$,
$P_n(a_n^{-1}\phi_n>\ep)=o(1)$. By assumption, it now follows
that $\phi_n\conv{Q_n}0$. Because $0\leq\phi_n\leq1$ the latter
conclusion is equivalent to $Q_n\phi_n=o(1)$.

Assume {\it (iv)}.
Let $\ep>0$ and 
$\phi_n:\scrX_n\rightarrow[0,1]$ be given. There exist $c>0$
and $N\geq1$ such that for all $n\geq N$,
\[
  Q_n\phi_n < c\,a_n^{-1}P_n\phi_n + \frac{\ep}2.
\]
If we assume that $P_n\phi_n=o(a_n)$ then there
is a $N'\geq N$ such that $c\,a_n^{-1}P_n\phi_n<\ep/2$
for all $n\geq N'$. Consequently, for every $\ep>0$, there exists
an $N'\geq1$ such that $Q_n\phi_n<\ep$ for all $n\geq N'$.

To show that {\it (ii) $\Rightarrow$ (iv)}, let $\mu_n=P_n+Q_n$ and
denote $\mu_n$-densities for $P_n,Q_n$ by $p_n,q_n:\scrX_n\rightarrow\RR$.
Then, for any $n\geq1$, $c>0$,
\begin{equation}
  \label{eq:aliensmall}
  \begin{split}
  \Bigl\|Q_n-&Q_n\wedge c\,a_n^{-1}P_n\|
    =\sup_{A\in\scrB_n}\Bigl( \int_Aq_n\,d\mu_n
      - \int_Aq_n\,d\mu_n \wedge \int_Ac\,a_n^{-1}p_n\,d\mu_n\Bigr)\\
    &\leq \sup_{A\in\scrB_n}
      \int_A (q_n-q_n\wedge c\,a_n^{-1}\,p_n)\,d\mu_n\\
    &= \int 1\{q_n>c\,a_n^{-1}p_n\}\,
      (q_n-c\,a_n^{-1}\,p_n)\,d\mu_n.
  \end{split}
\end{equation}
Note that the right-hand side of (\ref{eq:aliensmall}) is bounded above by
$Q_n(dP_n/dQ_n<c^{-1}a_n)$.

To show that {\it (iii) $\Rightarrow$ (iv)}, it is noted that, for all
$c>0$ and $n\geq1$,
\[
  0 \leq \int c\,a_n^{-1}P_n(q_n>c\,a_n^{-1}p_n)
    \leq Q_n(q_n>c\,a_n^{-1}p_n) \leq 1,
\]
so (\ref{eq:aliensmall}) goes to zero if
$\liminf_{n\rightarrow\infty} c\,a_n^{-1}P_n(dQ_n/dP_n>c\,a_n^{-1})=1$.

To prove that {\it (v)} $\Leftrightarrow$ {\it (ii)}, note that Prohorov's
theorem says that weak convergence of a subsequence within any
subsequence of $a_n(dP_n/dQ_n)^{-1}$ under $Q_n$
(see appendix~\ref{sec:defs}, {\it notation and conventions}) is
equivalent to the asymptotic tightness of
$(a_n(dP_n/dQ_n)^{-1}:n\geq1)$ under $Q_n$, \ie\ for every $\ep>0$
there exists an $M>0$ such that $Q_n(a_n(dP_n/dQ_n)^{-1}>M)<\ep$ for all
$n\geq1$. This is equivalent to {\it (ii)}. 
\end{proof}

\begin{proof} (proposition~\ref{prop:utfamily})\\
For every $\ep>0$, there exists a constant $\delta>0$ such that,
\[
  P_{\tht_0,n}\biggl(\,
    a_n\Bigl(\frac{dP_{\tht,n}}{dP_{\tht_0,n}}\Bigr)^{-1}(\samplen)
    >\frac1\delta\,\biggr)<\ep,
\]
for all $\tht\in B$, $n\geq1$. For this choice of $\delta$,
condition {\it (ii)} of lemma~\ref{lem:rcfirstlemma} is satisfied
for all $\tht\in B$ simultaneously, and \cf\ the proof of
said lemma, for given $\ep>0$, there exists a $c>0$ such that,
\begin{equation}
  \label{eq:limited}
  \|P_{\tht_0,n}-P_{\tht_0,n}\wedge c\,a_n^{-1}P_{\tht,n}\|<\ep,
\end{equation}
for all $\tht\in B$, $n\geq1$. Now note that for any $A\in\scrB_n$,
\[
  \begin{split}
  0\leq &P_{\tht_0,n}(A)-P_{\tht_0,n}(A)\wedge c\,a_n^{-1}P_n^{\Pi|B}(A)\\
  &\leq \int \bigl(
    P_{\tht_0,n}(A)-P_{\tht_0,n}(A)\wedge c\,a_n^{-1}P_{\tht,n}(A)
    \bigr)d\Pi(\tht|B).
  \end{split}
\]
Taking the supremum with respect to $A$, we find the following
inequality in terms of total variational norms,
\[
  \bigl\| P_{\tht_0,n}-P_{\tht_0,n}\wedge c\,a_n^{-1}P_n^{\Pi|B} \bigr\|
  \leq
  \int
    \bigl\| P_{\tht_0,n}-P_{\tht_0,n}\wedge c\,a_n^{-1}P_{\tht,n} \bigr\|
  d\Pi(\tht|B).
\]
Since the total-variational norm is bounded and $\Pi(\cdot|B)$ is a
probability measure, Fatou's lemma says that,
\[
  \begin{split}
  \limsup_{n\rightarrow\infty}
    \bigl\| &P_{\tht_0,n}-P_{\tht_0,n}\wedge c\,a_n^{-1}P_n^{\Pi|B}\bigr\|\\
  &\leq \int
    \limsup_{n\rightarrow\infty}
    \bigl\| P_{\tht_0,n}-P_{\tht_0,n}\wedge c\,a_n^{-1}P_{\tht,n} \bigr\|
  d\Pi(\tht|B),
  \end{split}
\]
and the \rhs\ equals zero \cf\ (\ref{eq:limited}). According to
condition~{\it(iv)} of lemma~\ref{lem:rcfirstlemma} this implies the
assertion.
\end{proof}

\begin{proof} (lemma~\ref{lem:rcsubset})\\
Fix $n\geq1$. Because $B_n\subset C_n$, for every $A\in\scrB_n$, we have,
\[
  \int_{B_n} P_{\tht,n}(A)\,d\Pi(\tht)\leq \int_{C_n}P_{\tht,n}(A)\,d\Pi(\tht),
\]
so $P_n^{\Pi_n|B_n} (A)\leq \Pi_n(C_n)/\Pi_n(B_n)\,P_n^{\Pi_n|C_n}(A)$.
So if for some sequence $\phi_n:\scrX_n\to[0,1]$, we have
$P_n^{\Pi_n|C_n}\phi_n(\samplen)=o(\Pi_n(B_n)/\Pi_n(C_n))$, then
the $P_n^{\Pi_n|B_n}$-expectations of $\phi_n(\samplen)$ are $o(1)$,
proving the first claim. If
$P_n^{\Pi_n|C_n}\phi_n(\samplen)=o(a_n\Pi_n(B_n)/\Pi_n(C_n))$, then
$P_n^{\Pi_n|B_n}\phi_n(\samplen)=o(a_n)$ and, hence,
$P_n\phi_n(\samplen)=o(1)$.
\end{proof}

\begin{proof} (theorem~\ref{thm:consistency})\\
Choose $B_n=B$, $V_n=V$ and use 
proposition~\ref{prop:prototype} to see that $P_n^{\Pi|B}\Pi(V|\samplen)$
is upper bounded by $\Pi(B)^{-1}$ times the \lhs\ of
(\ref{eq:bayesiantestingpower}) and, hence, is of order $o(a_n)$. Condition
{\it(ii)} then implies that $P_{\tht_0,n}\Pi(V|\samplen)=o(1)$,
which is equivalent to $\Pi(V|\samplen)\convprob{P_{\tht_0,n}}0$ since
$0\leq\Pi(V|\samplen)\leq1$, $P_{\tht_0,n}$-almost-surely, for all $n\geq1$.
\end{proof}

\begin{proof} (corollary~\ref{cor:schwartz})\\
A prior $\Pi$
satisfying condition~{\it (ii)} guarantees that $P_0^n\ll P_n^{\Pi}$ for
all $n\geq1$, \cf\ the remark preceding proposition~\ref{prop:dompriorpred}.
Choose $\ep$ such that $\ep^2<D$.
Recall that for every $P\in B(\ep)$, the exponential lower bound
(\ref{eq:lowerbndlik}) for likelihood ratios of $dP^n/dP_0^n$
exists. Hence the limes inferior of $\exp(\ft12n\ep^2)
(dP^n/dP_0^n)(\samplen)$ is greater than or equal to one with
$P_0^\infty$-probability one. Then, with the use of Fatou's lemma and the
assumption that $\Pi(B(\ep))>0$,
\[
  \liminf_{n\to\infty} \frac{e^{nD}}{\Pi(B)}\int_B
    \frac{dP_{\tht}^n}{dP_{\tht_0}^n}(\samplen)\,d\Pi(\tht)
  \geq 1,
\]
with $P_{\tht_0}^\infty$-probability one, showing that sufficient
condition~{\it (ii)} of lemma~\ref{lem:rcfirstlemma} holds. Conclude
that,
\[
  P_0^n \, \ctg \, e^{nD}\, P_n^{\Pi|B},
\]
and use theorem~\ref{thm:consistency} to see that
$\Pi(U|\samplen)\convprob{P_{\tht_0,n}}1$.
\end{proof}

\begin{proof} (theorem~\ref{thm:rates})\\
Proposition~\ref{prop:prototype} says that
$P_n^{\Pi_n|B_n}\Pi(V_n|\samplen)$ is of order $o(b_n^{-1}a_n)$.
Condition {\it(iii)} then implies that
$P_{\tht_0,n}\Pi(V_n|\samplen)=o(1)$, which is equivalent to
$\Pi(V_n|\samplen)\convprob{P_{\tht_0,n}}0$ since
$0\leq\Pi(V_n|\samplen)\leq1$, $P_{\tht_0,n}$-almost-surely
for all $n\geq1$.
\end{proof}

\begin{proof} (theorem~\ref{thm:coverage})\\
Fix $n\geq1$ and let $D_n$ denote a
credible set of level $1-o(a_n)$, defined for all $x\in F_n\subset\scrX_n$
such that $P_n^{\Pi_n}(F_n)=1$. For any $x\in F_n$, let $C_n(x)$ denote a
confidence set associated with $D_n(x)$ under $B$. Due to
definition~\ref{def:confcred}, $\tht_0\in\Tht\setminus C_n(x)$ implies
that $B_n(\tht_0)\cap D_n(x)=\emptyset$. Hence the posterior mass of
$B(\tht_0)$ satisfies $\Pi(B_n(\tht_0)|x)=o(a_n)$.
Consequently, the function
$x\mapsto1\{\tht_0\in\Tht\setminus C_n(x)\}\,\Pi(B(\tht_0)|x)$
is $o(a_n)$ for all $x\in F_n$. Integrating
with respect to the $n$-th prior predictive distribution and dividing by
the prior mass of $B_n(\tht_0)$, one obtains, 
\[
  \frac{1}{\Pi_n(B_n(\tht_0))}\int
    1\{\tht_0\in\Tht\setminus C_n\}\,\Pi(B_n(\tht_0)|X^n)\,dP_n^{\Pi_n}
      \leq \frac{a_n}{b_n}.
\]
Applying Bayes's rule in the form (\ref{eq:disintegration}),
we see that,
\[P_n^{\Pi_n|B_n(\tht_0)}
    \bigl(\tht_0\in\Tht\setminus C_n(\samplen)\bigr)= \int
      P_{\tht,n}\bigl(\tht_0\in \Tht\setminus C_n(\samplen)\bigr)
      \,d\Pi_n(\tht|B_n)\leq\frac{a_n}{b_n}.
\]
By the definition of remote contiguity, this implies asymptotic coverage \cf\
(\ref{eq:coverage}).
\end{proof}

\begin{proof} (corollary~\ref{cor:hellconf})\\
Define $a_n=\exp(-C'n\ep_n^2)$, $b_n=\exp(-Cn\ep_n^2)$, so that
the $D_n$ are credible sets of level $1-o(a_n)$, the sets $B_n$ of
example~\ref{ex:GGV} satisfy condition~{\it(i)} of
theorem~\ref{thm:coverage} and $b_na_n^{-1}=\exp(cn\ep_n^2)$ for
some $c>0$. By (\ref{eq:GGVrc}), we see that condition~{\it(ii)}
of theorem~\ref{thm:coverage} is satisfied. The assertion now
follows.
\end{proof}


\end{document}